\documentclass[leqno,11pt,letterpaper, english]{amsart}

\usepackage{amsmath,amssymb,amsthm,graphicx,url}
\usepackage{epstopdf}
\usepackage{color}
\usepackage{dsfont}
\usepackage{enumitem}

\usepackage{mathrsfs}

\newcommand{\xqedhere}[1]{%
    \rlap{%
         \hbox to#1{%
           \hfil
           \llap{%
               \ensuremath{\square}
           }%
       }%
   }%
}

\def\pasdegrille{\let\grille = \pasgrille}

\def\aat#1#2#3{
\divide \dimen1 by 48 \dimen3=\dimen1 \multiply \dimen1 by #1
\advance \dimen1 by -\dimen3 \divide \dimen1 by 101 \multiply
\dimen1 by 100 \divide \dimen2 by \count11 \multiply \dimen2 by #2
\setbox0=\hbox{#3}\ht0=0pt\dp0=0pt
  \rlap{\kern\dimen1 \vbox to0pt{\kern-\dimen2\box0\vss}}\dimen1= \wd1
\dimen2=\ht1}
\def\pasgrille{
\count12= \dimen1 \divide \count12 by 50 \divide \dimen2 by
\count12 \count11 =\dimen2 \ \divide \dimen1 by 48
\setlength{\unitlength}{\dimen1} \smash{\rlap{\ }} \dimen1= \wd1
\dimen2=\ht1 }
\def\grille{
\count12= \dimen1 \divide \count12 by 50 \divide \dimen2 by
\count12 \count11 =\dimen2 \ \divide \dimen1 by 48
\setlength{\unitlength}{\dimen1}
\smash{\rlap{\graphpaper[1](0,0)(50, \count11)}} \dimen1= \wd1
\dimen2=\ht1 }

\pasdegrille
\newcommand{\ud}{\,\mathrm{d}}

\newcommand{\Op}{{\operatorname{Op}_h}}

\newcommand{\Opc}{{\operatorname{Op}_1}}

\newcommand{\WFh}{\operatorname{WF}_h}

\newcommand{\WFsh}{\operatorname{WF}_{s,h}}

\newcommand{\be}{\begin{equation}}
\newcommand{\ee}{\end{equation}}

\newcommand{\cI}{\mathcal I}

\newcommand{\cB}{{\mathcal B}}

\newcommand{\cG}{{\mathcal G}}

\newcommand{\cM}{{\mathcal M}}
\newcommand{\cO}{{\mathcal O}}

\newcommand{\cS}{{\mathcal S}}
\newcommand{\cT}{{\mathcal T}}

\newcommand{\comp}{{\mathbb C}}
\newcommand{\N}{{\mathbb N}}
\newcommand{\R}{{\mathbb R}}

\newcommand{\G}{{\mathcal G}}

\newcommand{\CC}{{\mathbb C}}
\newcommand{\NN}{{\mathbb N}}

\newcommand{\RR}{{\mathbb R}}

\newcommand{\supp}{\operatorname{supp}}

\renewcommand{\Re}{\mathop{\rm Re}\nolimits}
\renewcommand{\Im}{\mathop{\rm Im}\nolimits}

\theoremstyle{plain}

\newtheorem{thm}{Theorem}

\newtheorem{prop}{Proposition}[section]
\newtheorem{cor}[prop]{Corollary}
\newtheorem{lem}[prop]{Lemma}

\theoremstyle{definition}

\newtheorem{defn}[prop]{Definition} 

\numberwithin{equation}{section}

\def\squarebox#1{\hbox to #1{\hfill\vbox to #1{\vfill}}}

\newcommand{\eps}{{\epsilon}}

\usepackage{amsxtra}

\ifx\pdfoutput\undefined
  \DeclareGraphicsExtensions{.pstex, .eps}
\else
  \ifx\pdfoutput\relax
    \DeclareGraphicsExtensions{.pstex, .eps}
  \else
    \ifnum\pdfoutput>0
      \DeclareGraphicsExtensions{.pdf}
    \else
      \DeclareGraphicsExtensions{.pstex, .eps}
    \fi
  \fi
\fi

\title[Gevrey WKB method]{Gevrey WKB method for Pseudodifferential Operators of real principal type}

\author[R. Lascar]{Richard Lascar}
\address{LJAD, Universit\'e C\^ote-d'-Azur, 28 Parc Valrose, 06028 Nice cedex, France }
\email{richard.lascar@univ-cotedazur.fr}

\author[I. Moyano]{Iv\'an Moyano}
\address{LJAD, Universit\'e C\^ote-d'-Azur, 28 Parc Valrose, 06028 Nice cedex, France }
\email{imoyano@unice.fr}

\usepackage{amssymb}
\usepackage{amsmath, amsthm, amsopn, amsfonts}

\begin{document}    

\begin{abstract}
In this paper we investigate the conjugation of Fourier Integral Operators (FIOs) associated to Gevrey phases and symbols and the corresponding semiclassical pseudodifferential operators (pdos) in the Gevrey class. We obtain an Egorov theorem compatible with Gevrey FIOs and real principal part pdos with Gevrey symbols. As a consequence, we obtain a justification of the usual microlocal WKB expansion for Gevrey pdos which are of real principal part at a point in the phase space, with the natural Gevrey subexponential asymptotics with respect to the semiclassical parameter.
\end{abstract}   

\maketitle

\tableofcontents

\section{Introduction}  
\label{sec:Introduction}

The use of Gevrey functions in the study of some linear and non-linear partial differential equations has experienced a significant developement in recent years. In particular, the use of Gevrey norms have been instrumental in the context of some non-linear Cauchy problems that may be ill-posed in the usual Sobolev context \cite{GerardVaretMasmoudi} or the study of some singular limits difficult to describe rigorously using other topologies \cite{MouhotVillani,BedrossianMasmoudiMouhot}. \par

Let us recall that, as usual, for $N \in \NN$ and $s>0$ a function $f\in C^{\infty}(\RR^N;\comp)$ is Gevrey of order $s$ (Gevrey-$s$ for short, or simply $\cG^s$)  whenever for every compact $K \subset \RR^N$ the following holds:
\begin{equation}
\exists C_{K} > 0, \quad \forall \alpha \in \NN^N, \qquad \vert \partial^{\alpha} f \vert_{L^{\infty}(K)} \leq C_K^{\vert \alpha \vert +1} \alpha !^s. 
\label{eq:defGevrey}
\end{equation} As can be seen from (\ref{eq:defGevrey}), the case $s=1$ coincides with the usual class of analytic functions, but the cases $s>1$ may contain smooth functions which are not analytic. Hence, the Gevrey class may be understood as a class of functions in between the analytic and the smooth context. \par

The use of pseudodifferential calculus associated to particular Gevrey classes, started by Boutet de Monvel and Kree in their seminal paper, \cite[Sect. 1]{BoutetKree}, has also been important in the past, for example in connexion with the study of propagation of Gevrey singularities of these pseudodifferential operators \cite{BLascar,LascarLascarMelrose} or the diffraction of waves around an obstacle \cite{LebeauGevrey3}. More recently, the study of the FBI transform in the Gevrey context has risen some interesting questions related to the quantization in the complex plane of symbols admitting only quasi-holomorphic extensions which in fact are not unique \cite{HitrikLascarSjostrandZerzeri} and thus introduces new aspects with respect to the analytic framework \cite{SjostrandAsterisque,HitrikSjostrand,MelinSjostrand}.

In this work we are concerned with the study of Gevrey pseudodifferential operators and the resulting WKB expansion in the Gevrey class. In what follows, we introduce some definitions of these Gevrey symbols and the related pseudodifferential operators and then state our main results. We work in a semiclassical framework (cf. \cite{Zworski}) involving a possibly small parameter $h>0$.

Let us note, for $n \in \NN$, $m\in \RR$ and $s>0$, the set $\cS_s^m(\RR^n \times \RR^n)$ of semiclassical Gevrey $s$ symbols of order $m$. Following \cite[Sect. 1]{BoutetKree} we write $a \in \cS^m_s$ if and only if for some $C>0$
\begin{equation}
\vert \partial_x^{\alpha} \partial_{\theta}^{\beta} a(x,\theta,h)  \vert \leq C^{1 + \vert \alpha \vert + \vert \beta \vert}  \alpha!^s \beta!^s h^{-m}
\label{eq:Gevrey symbol}
\end{equation} for all $(x,\theta)\in \RR^n \times \RR^n$ and $\alpha,\beta \in \NN^n$. Observe that this condition is more precise than the usual definition of the class $\cS^m_{1,0}$ (cf. for example (cf. \cite[Chapter 4]{Zworski})) because the constant $C$ in (\ref{eq:Gevrey symbol}) is uniform with respect to the multi-indices $\alpha,\beta$. As a consequence, we shall need a specific version of symbolic calculus adapted to the class $\cS^m_s$ (in particular Proposition \ref{lemma:composition} below). Finally, recall that a pdo $A $ is elliptic at a point $(x_0,\theta_0) \in \RR^n \times \RR^n$ whenever its symbol $a$ satisfies $|a(x_0,\theta_0)| > 0$.

We consider as usual (cf. \cite[Chapter 4]{Zworski}) semiclassical pseudo-differential operators (pdo for short) defined as suitable extensions of 
\begin{equation*}
a(x,hD,h)u(x) = \frac{1}{(2 \pi h)^n} \iint a(x,\theta,h) e^{\frac{i}{h}\theta \cdot (x-y)} u(y) \ud y \ud \theta, \qquad u \in \mathscr{S}(\RR^n),
\end{equation*} for a given symbol $a \in \cS^m_s$. One has $a(x,hD,h) = \Op(a)$ and also $\Op(a) = \Opc(a_h)$ with $a_h(x,\theta) = a(x,h \theta)$. We shall use the notation $\Op(a) \in \Psi^m_s$ for such a pdo. \par

\subsection{Motivations, hypothesis and main results}

The motivation of our work is twofold. On the one hand, the WKB method is known to hold true in the smooth ($\mathcal{C}^{\infty}$) and analytic categories, which raises the natural question of its vaidity with respect to the specific Gevrey asymptotics required in the Gevrey framework (in particular of the form (\ref{eq:Gevrey small})). Our Theorem \ref{thm:WKB} makes this point explicit at least microlocally. On the other hand, we obtain this microlocal WKB expansion as a consequence of Theorem \ref{thm:Egorov1}, which is a Gevrey version of the Egorov theorem, a central result in microlocal analysis, first published in \cite{Egorov}, having its own independent interest. We state next our main hypothesis and the context of our results. \par 

Let $P=P(x,hD_x,h)$ a semiclassic Gevrey$-s$ PDO of order zero in $\RR^n$, with $n\geq 2$. Let $p=p(x,\xi)$ be the principal symbol of $P$ and let $H_p$ be the associated Hamiltonian. We assume the following hypothesis. 
    \begin{description}
    \item[(H1)] $p$ is real, 
    \item[(H2)] $dp(x_0,\xi_0) \not = 0$, 
    \item[(H3)] and  $p(x_0,\xi_0) = 0 $.
    \end{description} If $P=P(x,hD_x,h)$ satisfies hypothesis \textbf{(H1)}, \textbf{(H2)} and \textbf{(H3)}, it is customary to say that $P$ is of \emph{real principal type} at the point $(x_0,\xi_0) \in \RR^n \times \RR^n$ (cf. \cite[Definition 3.1]{EgorovBook}). \par     
    
The class of real principal type operators enjoys an important property allowing to reduce a general pdo to a canonical form. This is known as Egorov theorem, which in the $C^{\infty}$ class states roughly the following (cf. \cite[Proposition 3.1]{EgorovBook} for instance, among other references as \cite[Chapter 8]{Zworski}, \cite[Section 62]{Eskin}, \cite{LernerEgorov}): If $A = A(x,D_x)$ and $B = B(x,D_x)$ are two pdo of real principal part with the same principal symbol, then there exists elliptic pdo $R=R(x,D_x)$ and $S=S(x,D_x)$ of order zero such that the conjugation $AR - SB$ is a smoothing operator. Of course, that this theorem is valid for a large class of pseudo-differential (possibly semiclassical) operators, including Gevrey symbols (see below for a precise definition). On the other hand, if we assume further regularity properties on the class of symbols at hand, for instance Gevrey regularity (see \ref{eq:defGevrey} for a definition), we may expect to get more precise information on the canonical transformations and the pseudo-differential calculus involved. In \cite{Gramchev} gave a first Gevrey version, but the result is not adapted to the semiclassical pdos and in particular to the WKB asymptotics. In this paper we aim at proving a Gevrey version of Egorov theorem which is sensitive to a small parameter $h$. We do this by using the classical WKB approach.

\subsubsection{Main results: Gevrey WKB expansion and Gevrey Egorov's theorem}

Our main result is a microlocal semiclassical WKB expansion for real principal part operators in the Gevrey setting compatible with the usual required asymptotics in the Gevrey framework.
    
    \begin{thm}[Gevrey WKB expansion] Let $n\geq 2$. Let $P=P(x,hD_x,h)$ be a semiclassical $\G^s$ pdo of order zero in $\RR^n$ of symbol $p=p(x,\xi)$ and assume that $P$ is of real principal type at a point $(x_0,\xi_0) \in \RR^n \times \RR^n$.
     Let $S \subset \RR^n$ be a real $\G^s-$hypersurface of $\RR^n$ transversal to $H_p$ at  $(x_0,\xi_0)$  . Let $\varphi \in \cG^s(\R^n)$ be given. Assume that 
     \begin{equation*}
     p(x,\varphi_x')= 0 \qquad \textrm{and} \qquad \xi_0 = \varphi_x'(x_0).
     \end{equation*}
     Under these hypothesis, one may solve the WKB problem near $x_0$, i.e.: If $a_0$ and $b$ are given symbols in $\cS^0_s(\RR^n)$, one may find some $a \in \cS^0_s(\RR^n )$ such that 
     \begin{equation*}
     \left\{ \begin{array}{cc}
     \frac{1}{h} e^{-i \frac{\varphi}{h} } P\left(  a e^{i \frac{\varphi}{h} }  \right)  - b   =  \cO_{\cG^s}(h^{\infty}),  &  \textrm{close to } x_0, \\
     a\vert_S = a_0. & 
     \end{array} \right.
     \end{equation*}
     \label{thm:WKB}
    \end{thm}

The smallness condition in Theorem \ref{thm:WKB} is defined in (\ref{eq:Gevrey small}). As mentioned before, we shall obtain Theorem \ref{thm:WKB} as a consequence of a particular version of Egorov's theorem in the Gevrey setting, according to our next result.

 \begin{thm}[Gevrey Egorov theorem]
     Let $P=P(x,hD_x,h)$ be a $\cG^s$ pdo with principal symbol $p=p(x,\xi)$. Assume that $P$ is of real principal type at some point $(x_0,\xi_0)\in T^* \RR^n.$ If  
     \begin{equation*}
     p(x_0,\xi_0)= 0, \quad \frac{\partial p}{\partial \xi}(x_0,\xi_0)   \not = 0, 
     \end{equation*} then $P$ is microlocally conjugate to $hD_{x_1}$ by $\cG^s$ FIOs. 
\label{thm:Egorov1}
\end{thm}

We give a precise definition of microlocal conjugation in Definition \ref{def:microlocal conjugation}, after having introduced the necessary objects for the reader's convenience.

\begin{thm}[Egorov's theorem in the Gevrey class using FBI transforms]
Let $P = P(y, D_y)$ be a $G^s-$differential operator of degree $m$ such that near $(y_0,\eta_0) \in T^* \RR^n \setminus \left\{ 0 \right\}$, the symbol $p(y,\eta)$ is of real principal type. Then, one can find a $G^s$-FBI transform as (\ref{eq:FBI}) and a canonical transform $\kappa$ such that  
\begin{equation*}
hD_{\Re x_1} \mathcal{T}u - h^m \mathcal{T} P(y,D_y) u = \cO_{\cG^{2s-1}}(h^{\infty}),
\end{equation*} with $\kappa(y_0,\eta_0) = (x_0, \xi_0)$, $\xi_0 = \frac{2}{i} \partial_x \phi(x_0)$.
\label{thm:Egorov}
\end{thm}

Observe that this theorem is valid for a large class of pseudo-differential (possibly semiclassical) operators, including Gevrey symbols (see below for a precise definition).

\subsection{Strategy and outline}

In Section \ref{sec:Stationary} we recall some non-stationary and stationary asymptotics with complex Gevrey phase and Gevrey symbols and review the symbolic calculus adapted to Gevrey pdo. We also discuss formal symbols and Carleson's theorem.

In Section \ref{sec:FIO} we introduce the class of Fourier integral operators associated to Gevrey phases and symbols that will be crucial in the proof of our main theorems.

In Section \ref{sec:ProofEgorov} we prove Theorem \ref{thm:Egorov1} by using a microlocal WKB expansion, involving the usual steps: reduction to an evolution equation, construction of a suitable phase as a solution to an eikonal equation and finally the construction of a suitable symbol by imposing a hierarchy of transport equations.

\subsection{Notation}
According to \cite{HitrikLascarSjostrandZerzeri}, in this work we shall say that, for $s\geq 1$ a given function $g$, depending on a small parameter $h>0$, is a $\G^s$-small remainder, and we note $g  = \cO_{\G^s}(h^{\infty})$, if there exists $C>0$ such that for all $\beta \in \N^m$,  
\begin{equation}
\vert   \partial^{\beta} g \vert \leq C^{1 + \vert \beta \vert } \beta!^s \exp \left( - \frac{h^{-\frac{1}{s}}}{C}  \right).
\label{eq:Gevrey small}
\end{equation} As usual, we use $\mathscr{S}(\R^n)$ to denote the space of Schwartz functions in $\R^n$.

\section{Stationary phase and symbolic calculus in the Gevrey setting}  
\label{sec:Stationary}

In this section we review the some results concerning the symbolic calculus in with Gevrey pseudodifferential operators, as some of them are difficult to find in the literature. We first review the non-stationary and stationary lemmas for Gevrey phases and symbols. As usual, this allows to give a sense to the composition of pseudodifferential operators and yield suitable asymptotic expansions for the resulting symbols. Of course, all the results reviewed below are classical in the $\mathscr{C}^{\infty}$ framework, according to classical references as \cite{Hormander,LernerBook,Eskin,Zworski}, but in the Gevrey context we need to justify that all small remainders have the particular form given in (\ref{eq:Gevrey small}), which sometimes necessitates some modifications with respect to the smooth or analytic cases.

\subsection{Non-stationary and Stationary phase method in the Gevrey setting}

We consider a symbol $a = a(x,y;h)$ in $\cS^{m_0}_s(\R^n \times \R^m)$, $m_0 \in \R$, for a small parameter $0 < h \leq 1$. Let $f = f(x,y)$ be a phase function of class $\G^s(\R^n \times \R^m)$. Consider 
\begin{equation}
\cI_f (h,y) = \int_{\R^n} e^{\frac{i}{h} f(x,y)} a(x,y;h) \ud x,   \qquad y \in \R^m,
\label{eq:stationary phase}
\end{equation} for a compactly-supported (in $x$) symbol $a$. The goal of this section is to adapt the usual non-stationary and stationary phase asymptotics to the Gevrey setting.

\subsubsection{Non-stationary phase lemma}

The following result is a non-stationary phase lemma adapted to the Gevrey asymptotics in $h$, which has an independent interest.

\begin{lem}[Non-stationary phase] Let $a = a(x,y;h)$ in $\cS^{m_0}_s(\R^n \times \R^m)$ be compactly supported and let $\supp a$ be its support. Assume $f \in \cG^s(\R^n \times \R^m)$ is such that 
\begin{enumerate}
\item $\Im f(x,y) \geq 0$,   
\item $f_x'(x,y) \not = 0$ for every $(x,y) \in \supp a$.
\end{enumerate} Then, $\cI_f (h, \cdot)$ is a $\G^s$-small remainder, i.e., there exists $C>0$ such that for all $\beta \in \N^m$,  
\begin{equation*}
\vert   \partial_y^{\beta} \cI_f (h, \cdot) \vert \leq C^{1 + \vert \beta \vert } \beta!^s \exp \left( - \frac{h^{-\frac{1}{s}}}{C}  \right).
\end{equation*}
\label{prop: Proof Non-stationary Phase}
\end{lem}

\begin{proof} As usual, as $\supp a$ is compact, upon using a suitable partition of the unity, we may reduce the result to a purely local situation. Thanks to the hypothesis on $f$, we may assume that $\supp a$ is a sufficiently small neighbourhood of a point around which it is possible to use a diffeomorphism $\kappa$ straightening the phase $f$. Indeed, if $\supp a$ is sufficiently small, let $(\tilde{x},\tilde{y}) = \kappa (x,y)$ with $\tilde{x} = (\tilde{x}_1, \tilde{x}')$ and let $\kappa$ be such that 
 \begin{equation*}
 \tilde{x}_1 = f(x,y), \qquad \tilde{x}' = x', \qquad \tilde{y} = y, \qquad \qquad \textrm{in }\supp a.
 \end{equation*} Observe that, as $\Im f(x,y) \geq 0$ and $f_{x_1}(x,y) \not = 0$, this mapping is indeed a local diffeomorphism. Upon changing variables, we find  
\begin{equation*}
 \cI_f (h,y) = \int_{\R^n} e^{\frac{i}{h} f(x,y)} a(x,y;h) \ud x = \int_{\R^n} e^{\frac{i}{h} \tilde{x}_1} \tilde{a}(\tilde{x}_1,\tilde{x}',y;h) \ud \tilde{x}_1 \ud \tilde{x}'. 
\end{equation*} Now, taking derivatives of any order with respect to $y$, the usual non-stationary phase lemma gives exponential decay on $h$ and the result follows (cf. for instance \cite[Lemma 3.14]{Zworski}).

\end{proof}

\subsubsection{Stationary phase lemma}
\label{sec: Stationary Phase}

We turn now to the stationary phase asymptotics in the Gevrey setting. Upon using a suitable partition of unity, we may consider a purely local situation. Let us pick a point $(x_0,y_0) \in \R^n \times \R^m$ such that the following  
\begin{align}
f \textrm{ is real-valued}, \label{eq:phase real} \\
f_x'(x_0,y_0) = 0, \label{eq:phase stationnaire} \\
\det \left(   f_{xx}''(x_0,y_0) \right) \not = 0, \label{eq:phase stationnaire non degeneree}
\end{align} hold. Consider $(x,y)$ sufficiently close to $(x_0,y_0)$ in $\R^n\times \R^m$. Then, there exists a function $y \mapsto \chi(y)$ of class $\G^s$ in a neighbourhood of $y_0$ such that 
\begin{equation}
f_x'(\chi(y),y) = 0, \qquad \chi(y_0) = x_0. 
\label{eq:local picture}
\end{equation} We have the following stationary phase asymptotics for locally defined symbols that will be used in Proposition \ref{lemma:composition}.

\begin{lem}[Stationary phase] Assume that $\supp a$ is close to $x_0$ so that (\ref{eq:local picture}), (\ref{eq:phase real}), (\ref{eq:phase stationnaire}) and (\ref{eq:phase stationnaire non degeneree}) hold.  Then, the symbol
\begin{equation}
b(y,h) =  e^{-(i/h)f(\chi(y),y)}  \int a(x,y;h)   e^{(i/h)f(x,y)}  \ud x
\end{equation} is of class $\cG^s$ near $y_0$ and enjoys the asymptotic expansion
\begin{equation}
b(y,h) \sim \left\vert   \det \left(  \frac{1}{2 i \pi h}  f_{xx}''(\chi(y),y) \right) \right\vert^{-\frac{1}{2}} e^{i \pi \frac{\sigma}{4}   } \sum_{j\geq 0} h^j L_{f,j,y}[a] (\chi(y),y), 
\label{eq:phase stationaire}
\end{equation} where each $L_{f,j,y}[a]$ is a differential operator of degree $2j$ in $x$ and $\sigma$ is the signature of the matrix $D^2_x f (x_0,y_0)$.
\end{lem}

\begin{proof}
Since $f$ is real-valued, we can use Morse lemma. By Taylor expansion of order $2$ around $(x_0,y_0)$ we find 
\begin{equation*}
f(x,y) = f(\chi(y),y) + \frac{1}{2} \langle Q(x,y)(x - \chi(y)),   x - \chi(y)  \rangle,
\end{equation*} whenever $(x,y)$ is close to $(x_0,y_0)$, where 
\begin{equation*}
Q(x,y) =  2 \int_0^1 (1 - t) f_{xx}'' \left(  \chi(y) + t (x - \chi(y)), y  \right) \ud t.   
\end{equation*} For $(x,y)$ close to $(x_0,y_0)$, $Q(x,y)$ is symmetric and invertible. Setting 
\begin{equation*}
Q_0 = f_{xx}''(x_0,y_0), 
\end{equation*} we observe there is a map $(x,y) \mapsto R(x,y)$ such that 
\begin{equation*}
Q_0 = R(x,y)^t Q(x,y) R(x,y), \qquad \textrm{ with } R(x_0,y_0) = Id.
\end{equation*} Consider the substitution 
\begin{equation*}
(x,y) \mapsto (\tilde{x}(x,y),y),
\end{equation*} with 
\begin{equation*}
\tilde{x}(x,y) = R^{-1}(x,y) (x - \chi(y)), \qquad \textrm{near } (x_0,y_0).
\end{equation*} Write 
\begin{equation*}
b(y,w) = e^{- \frac{i}{h} f(\chi(y),y)} I(y,\omega) = \int e^{\frac{i}{2h}  \langle Q_0 \tilde{x}, \tilde{x} \rangle } \tilde{a}(\tilde{x},y,\omega) \ud \tilde{x}, 
\end{equation*} for a $\G^s$ symbol $\tilde{a}$ with $\tilde{x}-\supp$ close to zero. Next, using Plancherel's formula 
\begin{equation*}
b(y,w) = \vert \det( \frac{1}{2i \pi h} Q_0) \vert^{-\frac{1}{2}} e^{ - \frac{i h}{2}  \langle Q_0^{-1}D,D  \rangle} \tilde{a}(0,y).
\end{equation*} Now, the asymptotic expansion (\ref{eq:phase stationaire}) follows from \cite[7.6.7]{Hormander} as a consequence of the usual asymptotics for quadratic phases with nonnegative real part.

\end{proof}

\subsection{Symbolic calculus with Gevrey symbols}

The calculus of Gevrey PDO has been established previously. Let $Q = \Op(q)$, $A = \Op(a)$, $q \in \cS^{m_0}_s$, $a \in \cS^{m}_s$. The composition $Q \circ A = \Op(q \circ a)$ can be written 
\begin{equation}
(q \circ a) (x,\xi) = \frac{1}{(2\pi h)^n} \iint  e^{-\frac{i}{h}(x-y)\cdot (\xi-\eta) } q(x,\eta) a(y,\xi) \ud y \ud \eta.
\label{eq:symbole composition integrale}
\end{equation}

\begin{prop} Let $Q = \Op(q)$, $A = \Op(a)$, $q \in \cS^{m_0}_s$, $a \in \cS^{m}_s$ be given. Then, the composition $Q \circ A = \Op(q \circ a)$ satisfies, for arbitrary $N \in \N$,
\begin{equation}
(q \circ a)(x,\xi) = \sum_{\vert \alpha  \vert < N} \frac{h^{ \vert  \alpha \vert }}{\alpha !} D_{\xi}^{\alpha}q(x,\xi)\partial_x^{\alpha} a(x, \xi) + r_N(q,a)(x,\xi),
\label{eq:Taylor}
\end{equation} where $r_N(q,a)$ is a $\cG^s$ symbol of order $m_0 + m - N$ given by 
\begin{equation*}
r_N(q,a)(x,\xi) = \frac{h^N}{i(2\pi h)^n}\frac{1}{i^N} \sum_{|\alpha| = N} \int_0^1 \frac{(1-\theta)^{(N-1)}}{(N-1)!}\iint_{\RR^{2n}} e^{\frac{y\cdot \eta}{ih}}(\partial_{\xi}^{\alpha} q)(x,\xi + \eta) \partial_{x}^{\alpha} a(x+\theta y,\xi) \ud \theta \ud y \ud \eta. 
\end{equation*}
\label{lemma:composition}
\end{prop}

We shall use (\ref{eq:Taylor}) with $N=1$ in the following sections. The proof relies on the following Lemma, which folllows the lines in \cite[Lemmas 4.1.2 and 4.1.5]{LernerBook}.

\begin{lem}
For every $t \in \RR^*$, let $J^t$ be the operator defined by 
\begin{equation*}
J^t b(x,\xi) = |t|^{-n} \iint_{\RR^n \times \RR^n} b(x + z,\xi + \zeta) e^{\frac{z\cdot \xi}{it}}  \frac{\ud z \ud \zeta}{(2\pi)^n}, \qquad b \in \cS^m_s(\RR^n \times \RR^n).
\end{equation*} Then, $J^t$ maps $\cS_s^{m'}(\RR^n \times \RR^n)$ into itself. Moreover, for any $N \in \N$, 
\begin{equation*}
(J^t b)(x,\xi) = \sum_{\vert \alpha \vert < N} \frac{t^{\vert \alpha \vert  }}{\alpha!} D_{\xi}^{\alpha} \partial_x^{\alpha} b(x,\xi)  + r_N(t)(x,\xi) 
\end{equation*} with $r_N(t) \in \cS^{m-N}_s$ and
\begin{equation*}
r_N(t)(x,\xi) = t^N \int_0^1 \frac{(1 - \theta)^{N-1}}{(N-1)!} J^{\theta t} \left(  (D_{\xi} \cdot \partial_x)^N b \right)(x,\xi) \ud \theta. 
\end{equation*}
\label{lem:Jt}
\end{lem}

\begin{proof}

Let $N \in \N$. By Taylor's expansion up to order $N$ in the exponential we have

\begin{equation*}
J^t = \sum_{k < N} \frac{t^k}{k!}(D_{\xi} \partial_x)^k  + \int_0^1 \frac{(1 - \theta)^{N-1}}{(N-1)!} J^{\theta t} \left(  (tD_{\xi} \partial_x)^N \right) \ud \theta. 
\end{equation*} Hence, $J^t$ maps $\cS_s^{m'}(\RR^n \times \RR^n)$ into itself and one has 
\begin{equation*}
(J^t b)(x,\xi) = \sum_{\vert \alpha \vert < N} \frac{t^{\vert \alpha \vert  }}{\alpha!}(D_{\xi}^{\alpha} \partial_x^{\alpha} b(x,\xi)  + r_N(t)(x,\xi) 
\end{equation*} with 
\begin{equation*}
r_N(t)(x,\xi) = t^N \int_0^1 \frac{(1 - \theta)^{N-1}}{(N-1)!} J^{\theta t} \left(  (tD_{\xi} \partial_x)^N b \right)(x,\xi) \ud \theta. 
\end{equation*}

\end{proof}

\begin{proof}[Proof of Proposition \ref{lemma:composition}]

From (\ref{eq:symbole composition integrale}), we may write 
\begin{equation*}
(q \circ a) (x,\xi) = \frac{1}{(2\pi h)^n} \iint  e^{-\frac{i}{h} z \cdot \zeta } q(x,\xi + \zeta) a(x + z,\xi) \ud z \ud \zeta,
\end{equation*} thanks to the changes of variables $x + z = y$ and $\xi - \eta = \zeta$. For $t \in \RR^*$ set
\begin{equation*}
J_0^t b = \frac{1}{(2 \pi \vert t\vert)^n} \iint b(z,\zeta) e^{\frac{1}{it}z\cdot \zeta} \ud z \ud \zeta, \qquad b = b(z,\zeta) \in \cS^{m'}_s(\RR^n \times \RR^n).
\end{equation*} Then, setting $C_{x,\xi} = q(x,\xi + \zeta) a(x+\zeta,\xi)$ one has
\begin{equation*}
(q \circ a) (x,\xi) = J_0^h C_{x,\xi}.
\end{equation*} So thanks to Lemma \ref{lem:Jt} we may write
\begin{equation*}
(q \circ a)(x,\xi) = J^h C_{x,\xi}(0,0) =  \sum_{\vert \alpha \vert < N} \frac{h^{\vert \alpha \vert  }}{\alpha!}D_{\xi}^{\alpha} q(x,\xi) \partial_x^{\alpha} a(x,\xi)  + r_N(h)(q,a)(x,\xi)
\end{equation*} where
\begin{equation*}
r_N(h)(q,a)(x,\xi) =  \int_0^1 \frac{(1 - \theta)^{N-1}}{(N-1)!} J^{\theta h} \left(  (tD_{\xi} \partial_x)^N C_{x,\xi}(0,0) \right)\ud \theta.
\end{equation*} Now, expanding the remainder
\begin{equation*}
r_N(q,a)(x,\xi) = \frac{1}{(2 \pi h)^n} \sum_{\vert \gamma \vert = N} \frac{h^N N!}{\gamma!} \int_0^1 \iint_{\RR^{2n}} \frac{(1-\theta)^{N-1}}{(N-1)!}e^{\frac{z\cdot \zeta}{ih}} D_{\xi}^{\gamma} q(x, \xi + \zeta) \partial_x^{\gamma}a(x + \theta z,\xi) \ud z \ud \zeta \ud \theta. \end{equation*} Using Gevrey stationary phase lemma one may prove that $r_N(q,a)$ is a Gevrey symbol of order $m_0 + m - N$ and thus, 
\begin{equation*}
(q \circ a)(x,\xi) = \sum_{\vert \alpha  \vert < N} \frac{h^{ \vert  \alpha \vert }}{\alpha !} D_{\xi}^{\alpha}q(x,\xi)\partial_x^{\alpha} a(x, \xi) + r_N(q,a)(x,\xi),
\end{equation*} which ends the proof.

\end{proof}

\subsection{Formal symbols and Borel lemma in the Gevrey setting}

In this section we introduce a notion of Gevrey formal symbol of a given order and give conditions ensuring the existence of a Gevrey symbol realizing a given sequence of formal symbols, based on Carleson's theorem \cite{Carleson}.

\begin{defn}[Formal symbols]
Let $m\in \R$. Let $(a_j)_{j\geq 0}$ be a sequence of symbols such that
\begin{equation*}
a_j \in \cS^{-j + m}_s(\RR^n \times \RR^n), \qquad \forall j \in \N.
\end{equation*} The sequence of $(a_j)_{j\geq 0}$ is said to be a formal $\cG^s$ symbol of degree $m$ if 
there exists $C>0$ such that  
\begin{equation}
\vert \partial_x^{\alpha} \partial_{\theta}^{\beta} a_j(x,\theta) \vert \leq C^{1 + \vert \alpha \vert + \vert \beta \vert + j } j!^s \alpha!^s \beta!^s h^{j-m}, \qquad \forall (x,\theta) \in \RR^n \times \RR^n,
\label{eq:condition symbole precisee}
\end{equation} for all $\alpha, \beta \in \N^n$ and $j\in \N$.
\label{eq:symbole formel}
\end{defn}

Carleson's theorem \cite{Carleson} allows us to realise Gevrey $s$ formal symbols into symbols. This result is a Gevrey Borel theorem as follows. 

\begin{thm} Let $m\in \R$ and let $(a_j)_{j\geq 0} \in \cS^{-j+m}$ be a formal $\cG^s$ symbol of degree $m$. Then, there exists $a \in \cS_s^m(\RR^n \times \RR^n)$ such that for any $N>0$, $\alpha, \beta \in \NN^n$ one has uniformly
\begin{equation*}
\vert \partial_x^{\alpha}  \partial_{\theta}^{\beta} ( a  - \sum_{j < N} a_j  )   \vert \leq C^{1 + \alpha + \beta + N} N!^s\alpha!^s \beta!^s h^{N-m}. 
\end{equation*}
\label{thm:Carleson}
\end{thm}

\begin{proof}
Consider the sequence $(a_j(x,\theta,h)h^{-j}j!)_{j\geq 0}$. This is a $(s+1)$-sequence in the Boutet-Kree sense, according to \cite{BoutetKree}. Using Carleson theorem we have that there exists $g(t,x,\theta, h) \in \cG^{s+1}(\overline{R}_+, \cS_s^m)$ such that for all $j\geq 0$
\begin{equation*}
\partial_t^j g(0,x,\eta,h) = j! a_j(x,\theta,h) h^{-j}.
\end{equation*} Setting $a(x,\theta,h) = g(t,x,\theta, h)\vert_{t = h > 0}$ for $h$ small, by Taylor's formula 
\begin{equation*}
\vert \partial_x^{\alpha}  \partial_{\theta}^{\beta} ( a  - \sum_{j < N} a_j  )   \vert \leq \sup_{\theta' \in [0,1]} \vert \vert   \partial_x^{\alpha}  \partial_{\theta}^{\beta} \partial_{t}^{j}  g(\theta' h, x,\theta, h)     \vert     \frac{h^{N}}{N!}. 
\end{equation*} As $g$ is a function in $\cG^{s+1}$ function w.r.t. $t$ and in $\cG^s$ w.r.t. $(x,\theta)$, one has for some $C'>0$
\begin{equation*}
\vert \partial_x^{\alpha}  \partial_{\theta}^{\beta} ( a  - \sum_{j < N} a_j  )   \vert \leq C'^{1 + \vert \alpha \vert + \vert \beta \vert +  }    N!^s \alpha!^s \beta!^s h^{N-m},
\end{equation*} so the theorem is proven.
\end{proof}

\subsection{Formal quasinorms in the class of Gevrey symbols}

In the set $\cS_s^m(\RR^n \times \RR^n)$ we may define, following \cite{LascarLascar}, a family of quasinorms (cf. \cite{BoutetKree}) of the form :
\begin{equation}
N_m(a,T)(x,\theta) = \sum_{(\alpha,\beta) \in \NN^d } \frac{h^m T^{\vert \alpha \vert + \vert \beta \vert}}{\alpha!^s\beta!^s} \vert  \partial_x^{\alpha} \partial_{\theta}^{\beta} a(x,\theta) \vert, 
\label{eq:quasinorms}
\end{equation} where $T>0$ is a fixed parameter. We also set 
\begin{equation*}
\overline{N_m}(a,T):= \sup_{(x,\theta)} N_m(a,T)(x,\theta).
\end{equation*} By Leibniz rule, one has  
\begin{equation*}
\overline{N}_{m+m'}(aa',T) \leq  \overline{N}_{m}(a,T) \overline{N}_{m'}(a',T),
\end{equation*} for any $a \in \cS_s^m(\RR^n \times \RR^n)$ and $a' \in \cS_s^{m'}(\RR^n \times \RR^n)$.

These quasinorms will be important in the proof of Theorem \ref{thm:Egorov1}.

\section{Fourier Integral Operators in the Gevrey setting}
\label{sec:FIO}

In this section we present some basic results concerning the specific class of Fourier Integral Operators (FIO for short) associated to Gevrey symbols and classes and how they relate to the class of Gevrey pdo.

\subsection{FIO with Grevey symbols and phases}

We work with a local class of Gevrey semiclassical Fourier integral operator defined in an analogous way to \cite[61.1]{Eskin} in the $h$-independent case and also to \cite{Zworski} for the $h$ dependent case in the smooth class. These are operators of the form 
\begin{equation}
Fu(x,h) = \frac{1}{(2\pi h)^n} \iint_{\RR^n \times \RR^n} e^{\frac{i}{h}(S(x,\eta) -y \cdot \eta)}a(x,\eta) u(y,h) \ud y \ud \eta, 
\label{eq:GevreyFIO}
\end{equation} where the phase function $S = S(x,\eta) \in \cG^s(\RR^n \times \RR^n)$ is such that $\det S''_{x,\eta} \not = 0$  and the symbol $a=a(x,\eta)$ is in some class $a \in \cS^m_s(\RR^n \times \RR^n)$. We will sometimes use the notation $F \in \cI^m_s(\RR^n)$ for such an operator.

For a FIO $F$ as above, define (cf. \cite[Section 61.1]{Eskin}) its adjoint $F^*$ by 
\begin{equation*}
\langle  F u, v  \rangle   =  \langle u, F^* v\rangle,
\end{equation*} so that
\begin{equation*}
F^*u(x,h) = \frac{1}{(2\pi h)^n} \iint_{\RR^n \times \RR^n} e^{\frac{i}{h}(S(y,\eta) - x \cdot \eta) )} \overline{a(y,\eta)} u(y,h) \ud y \ud \eta. 
\end{equation*} Then, we have a Gevrey version of \cite[Lemma 62.4]{Eskin}.

\begin{lem} Let $F \in \cI^m_s$ and let $F^*$ be its adjoint. Then, 
\begin{equation*}
FF^*u(x,h) = \frac{1}{(2\pi h)^n} \iint_{\RR^n \times \RR^n} e^{\frac{i}{h}(S(x,\xi) - S(y,\xi) )} a(x,\xi) \overline{a(y,\eta)} u(y,h) \ud y \ud \xi, \qquad \forall u \in \mathscr{D}'(\R^n), 
\end{equation*} we can write 
\begin{equation*}
FF^* = K_1 + K_2, 
\end{equation*} where $K_1 \in \Psi^{2m}_{s}$ and $K_2 = \cO_s(h^{\infty})$.
\end{lem}

\begin{proof}
Using a Gevrey$-s$ partition of the unity we may assume that $\supp(a)$ is small enough so that the mapping 
\begin{equation*}
\Sigma : (y,\eta) \mapsto (x,\xi), \qquad \xi = S'_x(x,\eta), \quad y = S'_{\eta}(x,\eta)
\end{equation*} is a canonical transformation on $\supp(a)$ of class $\cG^s$. Let $\chi \in \cG^s_0(\RR^n)$ and $r>0$ small enough so that $supp (\chi) \subset B(0,r)$ and $ \chi = 1$ on $B(0,\frac{r}{2})$. We can write 
\begin{equation*}
S(x,\xi) - S(y,\xi) = \Sigma(x,y,\xi)(x-y).
\end{equation*} Then, we have
\begin{equation*}
FF^* = K_1 + K_2
\end{equation*} with 
\begin{align}
K_1 u(x,h) & =  \iint_{\RR^n \times \RR^n} \mathfrak{a}(x,y,\eta) \chi\left( \frac{x-y}{\delta} \right)  e^{\frac{i}{h}(x-y)\cdot \eta} \vert J(x,y,\eta) \vert u(y,h) \frac{\ud y \ud \xi }{(2\pi h)^n}, \label{eq:k1} \\
K_2 u(x,h) & =  \iint_{\RR^n \times \RR^n} \mathfrak{a}(x,y,\eta) \left( 1 -  \chi\left( \frac{x-y}{\delta} \right) \right) e^{\frac{i}{h}(x-y)\cdot \eta} \vert J(x,y,\eta) \vert u(y,h)  \frac{\ud y \ud \xi}{(2\pi h)^n}, \label{eq:k2}
\end{align} where 
\begin{equation}
\mathfrak{a}(x,y,\eta) := a(x,\Sigma^{-1}(x,y,\eta)) \overline{a(y,\Sigma^{-1}(x,y,\eta))}, \qquad J = \partial_{\eta} \Sigma^{-1}.
\end{equation} Using Kuranishi's trick (cf. \cite[Chapter 8]{Zworski}), we deduce that $K_1$ is a $\cG^s$ PDO  of order $2m$. \par 

Next, using Gevrey non stationary phase we obtain that $K_2$ is a Gevrey $s$ negligeable reminder, while it is easy to compute the principal symbol of $K_1$, which is of order $2m$.

\end{proof}

\subsection{Composition of FIO and PDOs in the Gevrey setting}

The following result is an adaptation of \cite[Lemmas 62.2 and 62.3]{Eskin} to the Gevrey stting, which is possible with only minor modifications thanks to Lemma \ref{prop: Proof Non-stationary Phase}.

\begin{prop}
Let $F$ be a $\cG^s$-FIO of the form (\ref{eq:GevreyFIO}) for some phase $S\in \cG^s$ and symbol $a\in \cS^{m_0}_s$. Then, for every $A = A(x,D) \in \Psi_s^m $ one has 
\begin{equation}
AFu(x) = \frac{1}{(2\pi h)^n} \int c(x,\eta) e^{\frac{i}{h}S(x,\eta)} \hat{u}(\eta) \ud \eta, \qquad u \in \mathscr{D}'(\R^n),
\end{equation} for some $c \in \cS_s^{m + m_0}$ whose principal symbol satisfies
\begin{equation*}
c_0(x,\eta) = a_0(x,S'_x(x,\eta))a(x,\eta),
\end{equation*} where $a_0$ is the principal symbol of $A$. Also,
\begin{equation}
FAu(x) = \frac{1}{(2\pi h)^n} \int_{\R^{n} } c'(x,\eta) e^{\frac{i}{h}S(x,\eta)} \hat{u}(\eta) \ud \eta + R'u, \qquad u \in \mathscr{D}'(\R^n),
\end{equation} for some $c' \in \cS^{m + m_0}$ whose principal symbol satisfies 
\begin{equation}
c_0'(x,\eta) = a_0S'_\eta(x,\eta)a(x,\eta),
\end{equation} and $R' \in \Psi_s^{-\infty}$.
\label{prop:composition FIO pdo}
\end{prop}

\subsection{Gevrey wavefront and Lagrangian distributions}

We end this section with some details on the wavefront set adapted to the Gevrey setting.

\begin{defn}[Microlocal conjugation by Gevrey FIOs]
Given $P,Q \in \Psi^m_s$, we say that $P$ and $Q$ are microlocal conjugate in the $\cG^s$ sense around a point $(x_0,\xi_0) \in T^*\R^n$ if there exists a FIO $R \in \cI_s^m$ such that $RP$ equals $QR$ microlocally around $(x_0,\xi_0)$.
\label{def:microlocal conjugation}
\end{defn}

Consider a Lagrangian distribution of the form
\begin{equation}
u(x,h) = \int e^{\frac{i}{h} \phi(x,\theta) } a(x,\theta,h) \ud \theta
\label{eq:Lagrangian distribution}
\end{equation} for a phase $\phi$ and amplitude $a$ satisfying 
\begin{align*}
\phi \in \G^s(\R^n \times \R^N), \\
\Im \phi \geq 0, \qquad d\phi \not = 0, \\ 
a \in \cS^{m_0}_s(\R^n \times \R^N).
\end{align*}
Assume $\phi$ homogeneous of degree 1 in $\theta$ and $a \in \cS^{m_0,k_s}$ or $a \in \G^s_0 \cap \cS^{m_0}_s$ (see \cite{LascarLascar})

If $\WFh(u)$ is the semiclassical wavefront of $(u(\cdot, h))_h$, then it is standard that 
\begin{equation*}
\WFh(u) \subset \left\{  (x,\phi_x') \, \vert \quad (x,\theta) \in F, \, \phi_{\theta}'(x,\theta)  = 0           \right\},
\end{equation*} with $supp (a) \subset F$. The next result is a analogue version of this fact for the Gevrey semiclassical wavefront $\WFsh(u)$.

\begin{lem} Let $u$ as in (\ref{eq:Lagrangian distribution}). Then, one has 
\begin{equation*}
\WFsh(u) \subset \Lambda = \left\{  (x,\phi_x') \, \vert \quad (x,\theta) \in F, \, \phi_{\theta}'(x,\theta)  = 0           \right\},
\end{equation*}
\end{lem}

\begin{proof} This is a consequence of the results in \cite[Ch. 25]{Hormander} with minor modifications. The Gevrey asymptotic comes from the non-stationary phase lemma in Lemma \ref{prop: Proof Non-stationary Phase}. 

\end{proof}

\section{Proof of Theorem \ref{thm:Egorov1}}
\label{sec:ProofEgorov}

The goal of this section is to prove Theorem \ref{thm:Egorov1}. We divide the proof into several steps, according to the usual WKB method.

\par 
\vspace{0.5em}
\textbf{Step 1. Setting of the problem.}

We may assume w.l.g. that $(x_0,\xi_0)=(0,0)$. Let us choose once for all a particular real coordinate and write as usual $x =(x_1,x')\in \R\times \R^{n-1}$ and $\xi =(\xi_1,\xi')\in \R\times \R^{n-1}$. As $P$ is of real principal type by hypothesis, we can assume that the principal symbol of $P$ writes (cf. for instance \cite[Section 12.2]{Zworski}) 
\begin{equation}
p (x,\xi_1,\xi') = \xi_1 - \lambda(x,\xi'),
\label{eq:Thm2 evolution eq symbol}
\end{equation} for some real symbol $\lambda \in \cS_s^0(\R^n \times \R^{n-1})$ real. If $\lambda(x,\xi')$ stands for the principal symbol of $Q$ let us write 
\begin{equation}
P(x,hD_{x_1},hD_{x'}) = hD_{x_1} + Q(x,hD_{x'}).
\label{eq:Thm2 evolution eq}
\end{equation}  

\par 
\vspace{0.5em}
\textbf{Step 2. Suitable phase using the eikonal equation.}

We consider now the following Cauchy problem: Find $\varphi$ such that   
\begin{equation}
\left\{  \begin{array}{l}
\frac{\partial \varphi}{\partial x_1} - \lambda \left( x, \varphi_x'  \right) = 0, \\
\varphi|_{x_1=0} = x'\cdot \eta',
\end{array}  \right.
\label{eq:Thm2 eikonal} 
\end{equation} where $\lambda \in \cS_s^0(\R^n \times \R^{n-1})$ given by (\ref{eq:Thm2 evolution eq symbol}). The Cauchy problem (\ref{eq:Thm2 eikonal}) has a solution $\varphi$ of class $\cG^s$. Moreover, since $\det \varphi_{x',\xi'}'' \not = 0$ the phase function $\varphi$ generates a canonical transform. \par 

Our goal in the following sections is to find suitable $h$-FIOs $F,G$ such that 
\begin{equation}
GPF = hD_{x_1} + \cO_s(h^\infty), 
\label{eq:Thm2 claim}
\end{equation} 

\par 
\vspace{0.5em}
\textbf{Step 3. Reducing the order of the remainder by conjugation.}

Let us introduce a $h-$FIO of the form 
\begin{equation*}
Fu(x,h) = \frac{1}{(2 \pi h)^{n-1}} \int e^{\frac{i}{h}\varphi(x,\eta')} a(x,\eta') \hat{u}(x_1,\frac{\eta'}{h}) d \eta', \qquad u \in \mathscr{S}(\R^n),
\end{equation*} where $a \in \cS_s^{\mu}(\R^n \times \R^{n-1})$, $\hat{u}$ denotes the partial Fourier transform in the variables $x'$ only and $\varphi$ is given as before. Assume that $a$ is elliptic when $x_1 = 0, x' = 0, \xi_1 = 0, \xi'=0$.

Now, $F$ defined above is a h-FIO of the form (\ref{eq:GevreyFIO}) associated to a $\cG^s$ phase and symbol. As $P \in \Psi_s^0$ by hypothesis, we can compose $F$ and $P$ using Proposition \ref{prop:composition FIO pdo}. Indeed, from (\ref{eq:Thm2 evolution eq}) we get, for every $u \in \mathscr{S}(\R^n)$, 
\begin{equation*}
PFu = (hD_{x_1} + Q)Fu =  hD_{x_1}Fu + QFu 
\end{equation*} Now, as the symbol $a$ is supposed to be elliptic at the point $ (x_1, x', \xi_1, \xi')=0$, we may find a microlocal inverse of $F$ near this point and thus write  
\begin{equation*}
PFu = F(hD_{x_1}) + F_{-1} u, 
\end{equation*} for some $F_{-1} \in I_s^{\mu-1}$. Moreover, the symbol of $F_{-1}$, namely $\mathfrak{f}$, admits the expansion (as in \cite[Theorem 7.7.7]{Hormander}):
\begin{equation*}
\mathfrak{f} = \sum_{\alpha \geq 0} \frac{h^{|\alpha|}}{\alpha !} p^{(\alpha)}(x,\varphi_x') D_y^{\alpha}(e^{\frac{i}{h}\rho}a)|_{y=x},
\end{equation*} with 
\begin{equation*}
\rho(x,y,\eta')= \varphi(y,\eta') - \varphi(x,\eta') - (y-x) \varphi_x'(x,\eta'). 
\end{equation*} We observe that the first term of the expansion vanishes, for 
\begin{equation*}
 p(x,\varphi_x') (e^{\frac{i}{h}\rho}a)|_{y=x} =  p(x,\varphi_x') (e^{\frac{i}{h}(\varphi(y,\eta') - \varphi(x,\eta') - (y-x))}a)|_{y=x} = hD_{x_1} \varphi - \lambda(x,\varphi_{x'}') = 0, 
\end{equation*} thanks to (\ref{eq:Thm2 eikonal}). Hence, $F_{-1} \in \Psi_s^{-1}$.  \par 

Let now $G$ be the microlocal inverse of the elliptic $h$-FIO $F$ near zero. Then, we can write 
\begin{equation*}
GF = I + r, \qquad \textrm{microlocally near } (0,0,0,0),
\end{equation*} for $r \in \Psi_s^{-\infty}$ and hence
\begin{equation*}
GPFu = GF(hD_{x_1}) + GF_{-1}  = hD_{x_1} + R,   
\end{equation*} where $R = r hD_{x_1} + GF_{-1}$. Now, as $G$ and $F_{-1}$ are associated to the same generating function, we have in particular that $GF_{-1}$ is a pdo and moreover $GF_{-1} \in \Psi_s^{-1}$. Furthermore, as $r$ is a regularising pdo, we also have $r hD_{x_1} \in \Psi_s^{-1}$, which guarantees that 
\begin{equation*}
GPFu =  hD_{x_1} + R,   \qquad R \in \Psi_s^{-1},
\end{equation*} microlocally near zero. This means that $P$ is microlocally conjugate by $\cG^s$ FIOs to $hD_{x_1}$ up to a remainder of degree $-1$.

\par 
\vspace{0.5em}
\textbf{Step 4. Reducing arbitrarily the order of the remainder by conjugation.}

We aim now to iterate the scheme of the previous step in order to get (\ref{eq:Thm2 claim}) for a remainder of arbitrary order. We do this thanks to the following iterative scheme, which produces a formal Gevrey symbol.

\begin{lem} Let $q \in \cS^{-1}_s(\RR^n \times \RR^n)$. Then, there is a sequence $(a_j)_{j\geq 0}$ of symbols such that $a_j \in \cS^{-j}_s(\RR^n \times \RR^n)$ satisfying
\begin{equation*}
\frac{h}{i} \partial_{x_1}a_0 + qa_0 = 0, \qquad \textrm{and} \qquad a_0\vert_{x_1 = 0} \textrm{ is elliptic},
\end{equation*} and for all $j\geq 1$,
\begin{equation}
\frac{h}{i} \partial_{x_1}a_j + q a_j = r_1(q,a_{j-1}), \qquad \textrm{and} \qquad a_j\vert_{x_1 = 0} = 0.
\label{eq:stationary equations}
\end{equation} Moreover, if
\begin{equation*}
A_j = \Op(a_j), \qquad A^{(N)} = A_0 + \dots + A_{N-1}, \,\, N \geq 1,
\end{equation*} one has 
\begin{equation*}
(hD_{x_1} + Q) A^{(N)} = h A^{(N)}D_{x_1} + R_N, 
\end{equation*} with $R_N$ a $\cG^s \in \Psi_s^{-(N+1)}$. Finally, if $q$ is of order $m_0$ small enough the above sequence of $(a_j)_{j\geq 0}$ is a formal Gevrey $s$ symbol in the sense of Defintion \ref{eq:symbole formel}.
\label{lem:transport recursif}
\end{lem}

\begin{proof}
Equations (\ref{eq:stationary equations}) are solved by induction setting for $j\geq 1$
\begin{equation*}
a_j = -i \frac{a_0}{h} \int_0^{x_1} \frac{r_1}{a_0}(q,a_{j-1}) \ud t.
\end{equation*} One has $a_j \in \cS_s^{-j}(\RR^n \times \RR^n)$, in view of (\ref{eq:Taylor}) and (\ref{eq:stationary equations}) we have $R_N$ is of order $-N-1$.

\par
\vspace{0.5em}

We now show that if $q$ is of order $m_0$ small enough the above sequence of $(a_j)_{j\geq 0}$ is a formal Gevrey $s$ symbol.

 In order to prove (\ref{eq:condition symbole precisee}) we use the formal quasinorms for Gevrey symbols introduced in (\ref{eq:quasinorms}). Upon repeating the scheme of the previous step, we may assume that $P=hD_1 + Q$, with $Q \in \Psi_s^{m_0}$ of order $m_0$ small, is a microlocal conjugate to $hD_{x_1}$ by $\cG^s$ FIOs, modulo negligeable remainders. If $q$ is the symbol of $Q$, recall that 
\begin{equation*}
\Op(q) \circ \Op(a) = \Op(qa + R), \qquad R = R(a,q) 
\end{equation*} with a remainder $R$ as in Lemma \ref{lemma:composition}. Let us define next $\chi \in \cG^s_0(\RR^{2n})$ such that 
\begin{equation*}
\supp \chi \subseteq \left\{(y,\eta); \, \vert y \vert  +   \vert \eta \vert \leq r \right\} \qquad \textrm{ and } \qquad \chi = 1 \textrm{ on } \left\{(y,\eta); \, \vert y \vert  +   \vert \eta \vert \leq \frac{r}{2} \right\}
\end{equation*} and let us write 
\begin{equation*}
R(a,q) = R_{\chi}(a,q) + S_{\chi}, \qquad R_{\chi} = \chi R(a,q), \quad S_{\chi} = (1 - \chi) R(a,q).   
\end{equation*} We deduce that $S_{\chi}$ is Gevrey $s$ remainder by the non-stationary phase asymptotics in Lemma \ref{prop: Proof Non-stationary Phase}. Now if $\alpha,\beta \in \NN^n$, if $\eps>0$ small enough, we also have
\begin{equation}
\overline{N_m}(\partial_x^{\alpha}\partial_{\theta}^{\alpha}a,T) \leq  T^{- (\vert \alpha \vert + \vert \beta \vert ) }  \eps^{- s(\vert \alpha \vert + \vert \beta \vert ) }  \alpha!^s \beta!^s  \overline{N_m}(a,T(1 + \eps)^s).
\label{eq:seminorms estimate}
\end{equation} Now, if $q\in \cS^{m_0}_s(\RR^n \times \RR^n)$ with $m_0 \leq -n$, then from the previous equations, 
\begin{equation*}
\overline{N}_{m-1}(R_{\chi}(a),T) \leq \tau^{2n} C_n \eps^{-s} M_0(q) T^{-1} \overline{N}_m (a, T(1 + \eps)^s). 
\end{equation*} If $(a_j)_{j\geq 0}$ is the sequence defined by (\ref{eq:stationary equations}). We claim that estimate (\ref{eq:seminorms estimate}) implies, for $j\geq 0$, $\eps > 0$ small that for some $M_0' = M_0'(r,q,n)>0$ one has
\begin{equation}
\overline{N}_{-j}(a_j,T) \leq M_0'^j \eps^{-sj} j^{sj} \overline{N}_0(a_0,2T)T^{-j}.
\label{eq:seminorms estimate induction}
\end{equation} Indeed, this is true if $j=1$ from (\ref{eq:seminorms estimate}). Assuming that (\ref{eq:seminorms estimate induction}) holds for $j-1$ with $j\geq 2$, we have for $\eps$, $\tilde{\eps}>0$ small
\begin{equation*}
\overline{N}_{-j}(a_j,T) \leq \eps^{-sj} \frac{M_0'^j}{T^j} \frac{(j-1)^{s(j-1)}}{(1 + \eps)^{s(j-1)} \tilde{\eps}^{s(j-1)} } \overline{N}_0 \left(a_0,T(1 + \eps)^s\left( 1 + \frac{\tilde{\eps}}{j-1} \right)^{s(j-1)} \right). 
\end{equation*} We fix $\delta > 0$ small and choose $\frac{\tilde{\eps}}{j-1} = \frac{\delta}{j} $ and  $\eps = \frac{\delta}{j} $ in the previous inequality, so that 
\begin{equation*}
(1 + \eps)^s \left( 1 + \frac{\tilde{\eps}}{j-1} \right)^{s(j-1)} = \left( 1 + \frac{\delta}{j} \right)^{sj}. 
\end{equation*} It remains to check that 
\begin{equation*}
\frac{(1 + \eps)^{s(j-1)}}{ (j-1)^{s(j-1)} } \tilde{\eps}^{-s(j-1)} \delta^{sj} j^{?sj}\leq \frac{\delta^s}{j^s},
\end{equation*} which is $(1 + \frac{\delta}{j})^{-s(j-1)} \leq 1$, which is true since $j\geq 1$. So this allows to prove (\ref{eq:seminorms estimate induction}). \par

As a consequence, the sequence of symbols $(a_j)_{j\in \N}$ defined by (\ref{eq:stationary equations}) is a formal $\cG^s$ symbol as desired.

\end{proof}

\par 
\vspace{0.5em}
\textbf{Step 5. Conclusion.}

Thanks to Lemma \ref{lem:transport recursif}, there exists a sequence of symbols $(a_j)_{j\in \N}$ defined by (\ref{eq:stationary equations}) of order zero. Using Theorem \ref{thm:Carleson} we deduce that there exists $a \in \cS_s^m$ realising $(a_j)_{j\in \N}$. As a consequence, the $h$-FIO associated to the symbol $a$ and the phase $\varphi$ satisfying (\ref{eq:Thm2 eikonal}) satisfies that $PF = F(hD_{x_1})$ microlocally near zero up to negligeable remainders. This ends the proof of Theorem \ref{thm:WKB}.

\section{A Gevrey FBI version of Egorov's theorem: Proof of Theorem \ref{thm:Egorov}}

Let us describe a Gevrey$-s$ transform of the type 
\begin{equation}
\mathcal{T}u(x,h) = \int e^{\frac{i}{h}\varphi(x,y)} a(x,y,h) u(y) dy, \qquad u \in \mathcal{S}', 
\label{eq:FBI}
\end{equation} with some phase $\varphi$ of Gevrey class.

\subsection{Some complex symplectic geometry}

For $n \in \N$, we use coordinates $z = (z_1,\dots,z_n)$ for a point $z\in \CC^n$. As usual, we note $z_j = x_j + i y_j$, for $x_j,y_j \in \R$ and $j=1,\dots,n$, so that the holomorphic and antiholomorphic derivatives write

\begin{equation*}
\frac{\partial}{\partial z_j} = \frac{1}{2} \left( \frac{\partial}{\partial x_j} - i  \frac{\partial}{\partial y_j} \right), \qquad \frac{\partial}{\partial \overline{z}_j} = \frac{1}{2} \left( \frac{\partial}{\partial x_j} + i  \frac{\partial}{\partial y_j} \right).
\end{equation*} Let us denote  
\begin{equation*}
\cB = \left\{ \frac{\partial}{\partial z_j}, \frac{\partial}{\partial \overline{z_j}}  \right\}_{j=1,\dots,n}, \qquad \cB^* =   \left\{ \ud z_1, \dots, \ud z_n, \ud \overline{z}_1, \dots \ud \overline{z}_n \right\},
\end{equation*} where $\cB$ is the canonical complex basis of $\CC^n$ and $\cB^*$ is the associated basis of 1-forms. As usual, the canonical volume form is
\begin{equation*}
\ud z \wedge \ud \overline{z} = \left( \frac{2}{i} \right)^n \ud m(z),
\end{equation*} where $\ud m $ is the Lebesgue measure and 
\begin{equation*}
\ud z = \ud z_1 \dots \ud z_n, \qquad \ud \overline{z} = \ud \overline{z}_1 \dots \ud \overline{z}_n, 
\end{equation*}

If $\sigma$ be the canonical symplectic form on $\R^n \times \R^n$, given by 
\begin{equation*}
\sigma = \ud \xi \wedge \ud x. 
\end{equation*} the canonical form $\sigma_{\comp}$ in $\comp^n \times \comp^n$ such that
\begin{equation*}
\sigma_{\comp} = \ud \zeta \wedge \ud \eta,
\end{equation*} is an extension of $\sigma$ satisfying 
\begin{equation*}
\Re \sigma_{\comp} = \ud \xi \wedge \ud x - \ud \eta \wedge \ud y, \qquad \Im \sigma_{\comp} = \ud \xi \wedge \ud y + \ud \eta \wedge \ud x. 
\end{equation*} Recall that a submanifold $\Lambda \subset \comp^n \times \comp^n$ is totally real if $\Lambda \cap i \Lambda = \left\{ 0 \right\}$,  
 I-Lagrangian if $\Im \sigma_{\comp} \vert_{\Lambda} = 0 $ and R-Lagrangian if $\Re \sigma_{\comp} \vert_{\Lambda} = 0 $. If $A \in \cM_{n\times n}(\CC)$ and $f : \comp^n \rightarrow \comp$ is a quadratic function given by 
\begin{equation*}
f(z) = \frac{1}{2} \langle Az, z \rangle, \qquad z \in \comp^n,
\end{equation*} then the submainfold $\Lambda_f \subset \comp^n \times \comp^n$ defined by 
\begin{equation*}
\Lambda_f = \left\{ \left(z,\frac{\ud f}{\ud z}(z)  \right); \, z \in \comp^n   \right\}
\end{equation*} is Lagrangian for $\sigma_{\comp}$ (i.e., both I-Lagrangian and R-Lagrangian).

\subsubsection{FBI transforms}

Let $ \varphi = \varphi(z,x)$ be a holomorphic quadratic function on $\comp^n \times \comp^n$ satisfying that 
\begin{equation*}
\Im \left( \frac{\partial^2 \varphi}{\partial x^2} \right) \textrm{ is positive definite} \qquad \textrm{and} \qquad \det \left( \frac{\partial^2 \varphi}{\partial x \partial z} \right) \not = 0.
\end{equation*} Let $\kappa_{\varphi}$ the associated complex transformation $\kappa_{\varphi}: \CC^n \times \CC^n \rightarrow \CC^n \times \CC^n$ defined implicitly by 
\begin{equation*}
\kappa_{\varphi} \left(x, - \frac{\partial \varphi}{\partial x}(z,x)  \right) = \left( z, \frac{ \partial \varphi}{\partial z}(z,x)   \right), \qquad z,x \in \comp^n. 
\end{equation*} According to (\cite{Zworski}), the transformation $\kappa_{\varphi}$ is canonical for $\sigma_{\CC}$. The FBI operator associated to $\varphi$ is 
\begin{equation}
\cT_{\varphi} u(z) = \frac{c_{ \varphi}}{h^{\frac{3n}{4}}} \int_{\R^n} e^{\frac{i}{h} \varphi(z,x)} u(x)  \ud x, \qquad u \in \mathscr{S}(\R^n), 
\label{eq:FBI definition}
\end{equation} where 
\begin{equation*}
c_{\varphi} := \frac{\vert  \det \frac{\partial^2 \varphi}{\partial x \partial z}    \vert}{2^{\frac{n}{2}}  \pi^{\frac{3n}{4}}   \vert \det \Im \frac{\partial^2 \varphi}{\partial x^2} \vert^{\frac{1}{4}}  }.
\end{equation*} For the particular choices 
\begin{equation*}
\varphi_0(z,x) = \frac{i}{2} (z-x)^2, \qquad \Phi_0(z) = \frac{1}{2} \vert \Im z \vert^2, \qquad z,x\in \comp.
\end{equation*} we call $\cT_0$ the associated FBI transform defined by (\ref{eq:FBI definition}), which is called Bargman transformation.

We state the following corollary. 

\begin{cor}[Gevrey singular solutions] Instead of performing directly a WFB construction on $P$ we prefer to conjugate $P$ to $hD_1$ to solve the Cauchy problem 
\begin{equation*}
\left\{ \begin{array}{l}
 h D_1 u = 0, \\
  u\vert_{x_1 = 0} = u_0,
\end{array} \right.
\end{equation*} for which obviously we have $u = u_0 \otimes 1_{x_1}$.
\end{cor} Choosing a family $(u_0(x',h))$ with a Grevey $\WFh$ set reduced to one point in $T^*\RR^{n-1}$ we obtain a family $(u(x,h))$ of solutions to $Pu=0$ with a Gevrey $\WFh$ set reduced to a segment of the null bicharachteristic strip of $P$ through $(x_0,\xi_0)$.

\subsection{WKB Heuristics}

Set formally
\begin{equation*}
\mathcal{T}u(x,h) = \int e^{\frac{i}{h}\varphi(x,y)} a(x,y,h) u(y) dy. 
\end{equation*} The condition 
\begin{equation*}
hD_{\Re x_1} \mathcal{T}u - h^m \mathcal{T} P(y,D_y) u \sim 0 \qquad \textrm{in }G^{2s-1},
\end{equation*} is equivalent to solving an \textbf{eikonal} equation for the phase $\varphi$ of the form 
\begin{equation*}
\varphi_{x_1}'(x,y) = p(y, - \varphi_y'(x,y))
\end{equation*} and a transport equation for the symbol $a = a(x,y,h)$ of the form
\begin{equation*}
P'a = e^{-\frac{i}{h}\varphi } \left(  (hD_{\Re x_1}  - h^m P^t(y,D_y))(a e^{\frac{i}{h}\varphi})  \right) = 0, \quad \textrm{in } \cO_{s}(h^{\infty}),\textrm{close to }(x_0,y_0),
\end{equation*}

\subsection{Construction of the canonical transformation $\kappa$}

\begin{prop} Let $P = P(y, D_y)$ be a $G^s-$differential operator of degree $m$ such that near $(y_0,\eta_0) \in T^* \RR^n \setminus \left\{ 0 \right\}$, the symbol $p(y,\eta)$ is of real principal type. Then, there exists a canonical transform $\kappa$ such that  
\begin{equation*}
\kappa(y_0,\eta_0) = (x_0, \xi_0), \qquad \textrm{where} \quad \xi_0 = \frac{2}{i} \partial_x \phi(x_0).
\end{equation*} Moreover, if $y = \pi \circ\kappa^{-1}$, the submanifold defined by
\begin{equation*}
\Gamma = \left\{ (x,y(x)) \in \CC^n \times \RR^n; \,\, x \textrm{ close to } x_0    \right\}
\end{equation*} is totally real in $\CC^n \times \CC^n$ and of maximal dimension $2n$. 
\label{prop:kappa}
\end{prop}

\begin{proof}

Let $\chi$ be a real $G^s$ canonical map defined near $(y_0,\eta_0)$ mapping $p(y,\eta)$ to $\eta_1$. Such a $\chi$ can be obtained by applying Darboux Lemma. \par 
Now let $\mathcal{T}_0$ be the Bargman transform. $\mathcal{T}_0$ has phase functions $\varphi_0(x,y) = \frac{i}{2}(x-y)^2$, and $\Phi_0(x) = \frac{1}{2} (\Im x)^2$. \par 

Since $\xi \vert_{\Lambda_{\Phi_0}}$ is real, we will choose $\kappa = K_{\mathcal{T}_0} \circ \chi$. $\kappa$ is parametrized by $\in \CC^n$ close to $x_0 \in (0,y_0' - i \eta_0')$, setting 
\begin{equation*}
(y(x),\eta(x)) = \kappa^{-1}\left( x, \frac{2}{i} \partial_x \Phi_0(x) \right).
\end{equation*} The map $x \in \CC^n \mapsto y(x) \in \RR^n$ is of class $G^s$ and 
\begin{equation*}
\partial y : T_x \CC^n \rightarrow T_{y(x)}\RR^n \qquad \textrm{is surjective}.
\end{equation*} Moreover, $\overline{\partial} y$ is biejctive since $dx$ is canonical. \par 

Setting now 
\begin{equation*}
\Gamma = \left\{ (x,y(x)) \in \CC^n \times \RR^n; \,\, x \textrm{ close to } x_0    \right\}
\end{equation*} we get that $\Gamma$ is totally real in $\CC^n \times \CC^n$ and of maximal dimension $2n$.

\end{proof}

\subsection{Construction of the phase $\varphi$ on $\Gamma$}

In this section we obtain and solve an eikonal equation for the phase $\varphi$. The condition 
\begin{equation*}
hD_{\Re x_1} \mathcal{T}u - h^m \mathcal{T} P(y,D_y) u \sim 0 \qquad \textrm{in }G^s,
\end{equation*} is equivalent to solving an \textbf{eikonal} equation for the phase $\varphi$. We first construct $\varphi$ and $a$ on $\Gamma$. Next, we use an extension argument providing a (formal) sequence $(a_j)_{j\geq 0}$ of Gevrey symbols. We conclude by Carleson's moment method, which allow to construct a Gevrey symbol $a$ on $\CC^n \times \RR^n$.

\begin{prop}[Eikonal equation]
There exists a Gevrey-$s$ phase $\varphi = \varphi(x,y)$ satisfying
\begin{equation*}
\varphi_{x_1}'(x,y) = p(y, - \varphi_y'(x,y)) \qquad \textrm{ on } \Gamma
\end{equation*} for $\Gamma$ is given in Proposition \ref{prop:kappa}. Moreover, $\varphi_{y}(x_0,y_0)=-\eta_0 \in \RR^n$, $\operatorname{Im}\varphi''_{y}>0$, 
$\operatorname{det}\varphi''_{x,y}\neq0$
\label{prop:phase varphi}
\end{prop}

\begin{proof}
Let us choose $\varphi$ such that
\begin{equation*}
\frac{\partial \varphi}{\partial x} = \xi, \qquad \frac{\partial \varphi}{\partial y} = -\eta.
\end{equation*} Observe that, since $\kappa$ is canonical, the $1-$form
\begin{equation*}
\omega = \left(  \xi(x) - \frac{\partial y^t}{\partial x} \eta(x)  \right) dx - \frac{\partial y^t}{\partial \overline{x}} \eta(x) d\overline{x}
\end{equation*} is closed for $x$ close to $x_0$. Now, $\varphi$ defined above satisfies 
\begin{equation*}
\varphi_{x_1}' = \xi_1(x) = - \Im x_1 = p(y(x), \eta(x)) \qquad \textrm{and} \qquad \eta(x) = - \varphi_y'(x,y(x)).
\end{equation*} In order to get an FBI transform we may impose further 
\begin{equation*}
\varphi(x,y) = \frac{i}{2}(x'-y')^2 - y_{1,0}\eta_{1,0} + i C(y_1 - y_{1,0})^2, \qquad \Re C > 0.
\end{equation*}
\end{proof}

\subsection{Construction of the symbol $a$ on $\Gamma$}

\begin{prop}[Transport equation]
There exists a symbol $a = a(x,y,h)$ such that 
\begin{equation*}
P'a = e^{-\frac{i}{h}\varphi } \left(  (hD_{\Re x_1}  - h^m P^t(y,D_y))(a e^{\frac{i}{h}\varphi})  \right) = 0, \qquad \textrm{close to }(x_0,y_0).
\end{equation*} Furthermore, $a$ is of class $G^s$ and elliptic at $(x_0,y_0)$. 
\label{prop:symbol a}
\end{prop}

\begin{proof}
As $P'$ is a $h-PDO$ of degree $0$ with symbol 
\begin{equation*}
p'(x,y,\xi,\eta) = \Re \xi_1 + \varphi_{x_1}' - P\varphi(x,y,\eta), 
\end{equation*} where 
\begin{equation*}
P\varphi(x,y,\eta) = p(y,-\eta - \varphi_y'). 
\end{equation*}

According to Proposition \ref{prop:phase varphi} phase solves
\begin{equation*}
\varphi_{x_1}'(x,y) = p(y,-\varphi_y'(x,y)) \qquad \textrm{ on } \Gamma.
\end{equation*} Then for $(x,y) \in \Gamma$ such that $\xi = \eta = 0$ one has $p'=0$. Moreover, 
\begin{equation*}
a = \frac{\partial p'}{\partial \xi}\vert_{\xi = \eta = 0} = e_1, \qquad b = \frac{\partial p'}{\partial \eta}\vert_{\xi = \eta = 0} = \frac{\partial p}{\partial \eta}(y, -\varphi_y'),\qquad \textrm{on } \Gamma.
\end{equation*} Consider the vector field on $\CC^n \times \CC^n$ defined by 
\begin{equation*}
\mathcal{V}' = a \partial_x + \overline{a} \partial_{\overline{x}} + b \partial_{ y } + \overline{b} \partial_{\overline{y}}.
\end{equation*} One checks that $\mathcal{V}'\vert_{\Gamma}$ is tangent to $\Gamma$ if and only if 
\begin{equation*}
b \in \RR^n \qquad \textrm{and} \qquad  b = \left( \frac{\partial y}{\partial x}\right)a + \left( \frac{\partial y}{\partial \overline{x}}\right) \overline{a}.
\end{equation*} Observe that if $\mathcal{V} =  \alpha \partial_x + \overline{\alpha} \partial_{\overline{x}}$ is a vector field on $\CC^n$ and $u(x) = u'(x,y(x))$, one has $\mathcal{V}'u' = \mathcal{V}u$ on $\Gamma$ if and only if 
\begin{equation*}
\alpha = a, \qquad b = \left( \frac{\partial y}{\partial x}\right)a + \left( \frac{\partial y}{\partial \overline{x}}\right) \overline{a}, \quad \textrm{with} \quad \partial_{\overline{x}} u' = \partial_{\overline{y}}u' = 0 \textrm{ on }\Gamma.
\end{equation*} We fulfill the tangency condition on $\Gamma$ by choosing
\begin{equation*}
b = \frac{\partial p}{\partial \eta} (y, - \varphi_y') \qquad \textrm{and} \qquad a = \frac{\partial p'}{\partial \xi} \vert_{\xi = \eta = 0}. 
\end{equation*} This follows by construction, as by setting 
\begin{equation*}
\chi'H_p = H_{\eta_1} = \begin{pmatrix}  e_1 \\ 0 \end{pmatrix}, \qquad (\chi')^{-1} =  \begin{pmatrix}  P & Q \\ R & S \end{pmatrix} \textrm{(real symplectic matrix)},
\end{equation*} with $P$ invertible, $P^t R$ and $Q^t S$ symmetric and $P^t S - R^t Q = I$, we have
\begin{equation*}
\frac{\partial y}{\partial x} = \frac{1}{2} (P - iQ), \qquad \frac{\partial y}{\partial \overline{x}} = \frac{1}{2} (P + iQ).
\end{equation*} Then, 
\begin{equation*}
\left(  \frac{\partial y}{\partial x}e_1 +  \frac{\partial y}{\partial \overline{x}}e_1   \right) = Pe_1 = \frac{\partial p}{\partial \eta} (y, - \varphi_y'), \qquad \textrm{ on } \Gamma, 
\end{equation*} since by construction
\begin{equation*}
\chi^{-1}(\Re x, - \Im x) = (y(x), \eta(x)).
\end{equation*}

\end{proof}

We will solve transport equations on $\Gamma$ ad extend solutions almost holomorphycally from $\Gamma$. This yields a formal solution $(a_j')_{j\geq 0}$ with $a_j' \in \cS_s^{-j}(\comp^n \times \comp^n)$. \par 

We set 
\begin{equation*}
b' e^{\frac{i}{h}\varphi} = \frac{1}{h} \left(  hD_{\Re x_1} - h^m P^t(y,D_y)   \right)(a e^{\frac{i}{h}\varphi} ).
\end{equation*} The eikonal equation is satisfied since $\varphi_{\Re x_1}' = p(y,-\varphi_y')$. Now set 

\begin{equation*}
e^{\frac{i}{h}\varphi}P_{\varphi}(x,y,hD_y) = h^m P^t (y, D_y)(\cdot e^{\frac{i}{h}\varphi} ), 
\end{equation*} The principal symbol of $P_{\varphi}$, namely $p_{\varphi}$, satisfies
\begin{equation*}
p_{\varphi}|_{\eta = 0} = p(y,-\varphi_y').
\end{equation*} If 
\begin{equation*}
P' = h D_{\Re x_1} + \varphi_{\Re x_1}' - P_{\varphi}(x,y,hD_y)
\end{equation*} we have to solve a WKB problem for $P'$ with $\varphi'=0$, which leads to some expansions. If $Q(y,hD_y)$ is a $hPDO$ of order zero on $\R^n$, if $\varphi$ is a complex phase with $\Im \varphi \geq 0$ and $\ud \varphi \not = 0$ when $\Im \varphi = 0$, one has 
\begin{equation}
Q(a e^{\frac{i}{h}\varphi}) \sim e^{\frac{i}{h}\varphi} \sum_{\alpha \geq 0} \frac{h^{|\alpha|}}{\alpha !} \tilde{q}^{(\alpha)}(y,\varphi_y') D^{\alpha}_z(e^{\frac{i}{h}\rho} a)|_{y=z}, 
\label{eq:MelinSjostrand}
\end{equation} where $\tilde{q}$ is an almost holomorphic extension of $q$ to the whole $\comp^n \times \comp^n$, $a$ is a symbol and $\rho$ is given by 
\begin{equation*}
\rho(y,z) = \varphi(z) -  \varphi(y) - (z-y) \cdot \varphi_y'(y).
\end{equation*} The expansion (\ref{eq:MelinSjostrand}) was obtained by Melin and Sj\"ostrand in \cite{MelinSjostrand}. We deal with the Gevrey case and explain how to adapt the original proofs to compute the remainders in our case. \par 

We deal with phases of the form 
\begin{equation*}
a(y,X) = \varphi(z) + (y-z) \cdot \eta, 
\end{equation*} where we have written $X = (z, \eta) \in \R^{2n}$, $\varphi(z)$ is a phase function in $\cG^s$ defined close to $y_0 \in \R^n$ and complex-valued with $\Im \varphi \geq 0$, $\ud \varphi(y_0) = - \eta_0 \in \R^n \setminus 0$. We shall work close to $X_0 := (y_0,-\eta_0)$ and one has then 
\begin{equation*}
\Im a(y,X) \geq 0, \qquad a_X'(y_0,X_0) = 0, \qquad \det a_{X,X}''(y_0,X_0) \not = 0.
\end{equation*}

We note $u_h(X)$ some $\cG^s$ symbol on $\R^{2n}$ of degree zero compactly supported close to $X_0 = (y_0,\varphi'_y(y_0))$. We have the following lemma. 

\begin{lem}
Under the previous assumptions there is a stationary phase expansion of Gevrey type $2s-1$ with the form 
\begin{equation}
\iint_{\R^n \times \R^n} e^{\frac{i}{h}a(y,X) } u_h(X) \ud X \sim e^{\frac{i}{h} a(y,Z_y)} \sum_{\nu = 0}^{\infty} h^{\nu + n} C_{y,\nu}(D)u_h(Z_y),
\label{eq:MelinSjostrandGevrey}
\end{equation} where $\sim$ means modulo some small $\cG^{2s-1}$ remainders.
\label{lem:MelinSjostrandGevrey}
\end{lem} Above we have noted $a(y,Z)$ an almost holomorphic extension of $a(y,X)$ to the whole $\comp^{2n}$ of the same class as $a$ and similarly for $u_h$. Moreover, $y \mapsto Z_y$ is the function defined by $a_Z'(y,Z) = 0$ close to $X_0$ with $Z_{y_0} = X_0$.

\begin{proof}

We follow closely \cite{MelinSjostrand} and use Morse type coordinates. We also use the Stokes formula and deform $\R^{2n}$ into complex contours in $\comp^{2n}$. \par 

The equations 
\begin{equation*}
a_Z'(y,Z) = 0, \qquad Z_{y_0} = X_0
\end{equation*} defines a $\cG^s$ function since $\det a_{Z,Z}''(y_0,X_0) \not = 0$. By Lemma \cite[Lemma 2.1]{MelinSjostrand} one has 
\begin{equation*}
\Im a(y,Z_y) \geq \frac{1}{C} |\Im Z_y|^2.
\end{equation*} Set
\begin{equation*}
h(y,Z) = a(y,Z+Z_y) - a(y,Z_y),
\end{equation*}which is defined close to $(y_0,0)$ and modulo a small error induces in $Z$ a quadratic form. One has 
\begin{equation*}
\partial_Z h(y,0) = 0, \qquad \partial_{\overline{Z}} h(y,Z) = \cO(1) \exp( - \frac{1}{C} |\Im(Z + Z_y)|^{\frac{-1}{s-1}}  ).
\end{equation*} Set 
\begin{equation*}
R(y,Z) = 2 \int_0^1 (1 - \theta) h_{Z,Z}''(y,\theta Z) \ud \theta.
\end{equation*} Since $u \mapsto \exp( - \frac{1}{C} u^{\frac{-1}{s-1}}  )$ is increasing one has 
\begin{equation*}
\partial_{\overline{Z}} h(y,Z) = \cO(1) \exp( - \frac{1}{C} (   |\Im Z |^{\frac{-1}{s-1}} + |\Im Z_y|^{\frac{-1}{s-1}}   )    )
\end{equation*} Writing 
\begin{equation*}
R(y,Z )= i Q^t(y,Z )Q(y,Z ), \qquad Q(y_0,0 ) = A^{-1}, \quad A^t R(y_0,0 )A = iId,
\end{equation*} we define a change of coordinates in $\comp^{2n}$ by
\begin{equation*}
\overline{Z}(Z ) = Q(y,Z - Z_y )(Z - Z_y ),
\end{equation*} since $Q(y_0,0 ) \in GL(2n,\comp )$. In these coordinates one has 
\begin{equation*}
a(y,Z ) = a(y,Z_y ) + \frac{i}{2} \langle \overline{Z}, \overline{Z} \rangle + \rho(y,Z ),
\end{equation*} with 
\begin{equation*}
 \langle \overline{Z}, \overline{Z} \rangle := |\overline{X}|^2 - |\overline{Y}|^2 + 2i \overline{X} \cdot \overline{Y}, \qquad \overline{Z} = \overline{X} + i \overline{Y} \in \comp^{2n}. 
\end{equation*} The function $\rho$ is small and
\begin{equation*}
|\rho(y,Z )| = \cO(1 ) \exp (-\frac{1}{C} |\Im ( Z - Z_y )|^{-\frac{1}{s-1}}   ).
\end{equation*} The map $Z\mapsto \overline{Z}(Z )$ has an inverse $\overline{Z} \mapsto \overline{Z}(Z )$ and these two maps depend on $y$ also. Define for $\sigma \in [0,1]$ the integration paths 
\begin{equation*}
\Gamma_{y,\sigma} : \overline{X} \mapsto Z(\overline{Z}_{\sigma} ), \qquad \overline{Z}_{\sigma} = \overline{X} + i \sigma g(y,\overline{X} ), \quad \overline{X} \in \R^{2n},
\end{equation*} where $g$ is smooth and such that $\R^{2n}$ is given by $\overline{Y} = g(y,\overline{X} )$ in the new coordinates. Now, using Stokes formula we deform $\Gamma_{y,1} = \R^{2n}$ into $\Gamma_{y,0}$. Following \cite[Lemma 2.4]{MelinSjostrand}, we have 
\begin{equation*}
\Im a(y,Z(\overline{Z}_{\sigma} ) \geq \frac{1}{C}(1-\sigma ) ( |\Im Z_y|^2  + |\overline{X}|^2  ) \geq \frac{1}{C'} |\Im Z(\overline{Z}_{\sigma}|^2.
\end{equation*} Define the $(2n,0 )$ form in $\comp^{2n}$ by 
\begin{equation*}
\omega_h = f_h(Z ) \ud Z_1 \wedge \cdots \wedge \ud Z_{2n}, \qquad f_h(Z ) = e^{\frac{i}{h} a(y,Z )} u_h(Z ),
\end{equation*} since 
\begin{equation*}
\partial_{\overline{Z}} f_h = e^{\frac{i}{h} a(y,Z )} \left( \partial_{\overline{Z}} u_h + \frac{i}{h}u_h \partial_{\overline{Z}} a \right),
\end{equation*} we have 
\begin{equation*}
\iint_{\R^{2n}} e^{\frac{i}{h} a(y,X ) } u_h(X ) \ud X  = \iint_{\Gamma_{y,1}} e^{\frac{i}{h} a(y,Z ) } u_h(X ) \ud Z_1 \wedge \dots \wedge \ud Z_{2n}
\end{equation*} and 
\begin{equation}
 \iint_{\Gamma_{y,0}} e^{\frac{i}{h} a(y,Z ) } u_h(X ) \ud Z_1 \wedge \dots \wedge \ud Z_{2n} = \iint_{V} e^{\frac{i}{h} a(y, Z_y ) } u_h( Z(\overline{X} ) e^{-\frac{1}{2h} |\overline{X}|^2 + \frac{i}{h} \rho  } \left|\frac{\partial Z}{\partial \overline{X}}  \right| \ud \overline{X},  \label{eq:256}
\end{equation} where $V$ is a neighbourhood of the origin in $\R^{2n}$. \par 

Using \cite[Lemma 2.5]{MelinSjostrand} we may reduce the computation of the asymptotics of the right-hand side of (\ref{eq:256}) to the case $\rho = 0$. Indeed, if $Y \mapsto (AY,Y )$ is a non-degenerate quadratic form with $\Im A \geq 0$, and $\chi$ is a $\cG^s$ cut-off function close to the origin in $\R^N$ we can write 
\begin{equation*}
h^{-\frac{N}{2}} \int e^{\frac{i}{2h} (AY,Y ) } u_h(X + Y ) \ud Y = \ell_{\chi}(X ) + r_{\chi}(X), 
\end{equation*} for 
\begin{equation*}
\ell_{\chi}(X ) = h^{-\frac{N}{2}} \int e^{\frac{i}{2h} (AY,Y ) } u_h(X + Y ) \chi(Y ) \ud Y, \qquad  r_{\chi}(X) =  h^{-\frac{N}{2}} \int e^{\frac{i}{2h} (AY,Y ) } u_h(X + Y ) (1 - \chi(Y ) ) \ud Y.
\end{equation*} 
Now,as 
\begin{equation*}
| \partial_X^{\alpha} r_{\chi}(X )| \leq C^{1+|\alpha|} \alpha!^s \exp( - \frac{1}{C} h^{-\frac{1}{s}}  ),
\end{equation*} we deduce that $r_h$ is a small $\cG^s$ remainder. On the other hand, $\ell_X$ is a $\cG^s$ symbol of the same order as $u_h$ having the asymptotic expansion 
\begin{equation*}
\ell_{\chi}(X ) \sim C_A \sum_{\nu = 0} \frac{h^{\nu}}{\nu!(2i )^{\nu}} (A^{-1}D,D )^{\nu} u_h(X ), \qquad C_A = (\frac{1}{2\pi i}\det A )^{-\frac{1}{2}}.   
\end{equation*} As a consequence, 
\begin{equation*}
v_h(y ) := e^{-\frac{i}{h} a(y,Z_y ) } \int_{\R^{2n}} e^{\frac{i}{h} a(y,X ) }u_h(X ) \ud X 
\end{equation*} is a $\cG^s$ symbol of order $n$ admitting an expansion of the form 
\begin{equation*}
v_h(y ) \sim \sum_{\nu = 0}^{\infty} h^{\nu + n} C_{\nu,y}(D )u_h(Z_y )
\end{equation*} with $C_{\nu,y}(D )$ differential operators of degree less or equal than $2\nu$. The first term above rewrites as $h^nC_0(y ) u_h(Z_y )$ for a function $C_0(y )(2\pi )^{-n}$ which is a suitable branch of the square root of $\det(\frac{1}{i} a_{Z,Z}''(y,Z_y )  )^{-1}$. From \cite{HitrikLascarSjostrandZerzeri} we deduce that the term 
\begin{equation*}
R_N(y ) :=  \sum_{\nu = N}^{\infty} h^{\nu + n} C_{\nu,y}(D )u_h(Z_y )
\end{equation*} satisfies 
\begin{equation*}
|\partial_y^{\gamma} R_N(y ) | \leq C^{1+\gamma + N} \gamma!^s N!^{2s-1} h^{N+n},
\end{equation*} where the Gevrey loss is an standard observation. Moreover the last term coming from Stokes formula is also a small $\cG^s$ reminder at least when $a(y,\cdot )$ depends analytically in $y$, which is the case for the phase $a(y,X ) = (y-z )\eta + \varphi(z )$, $X=(z,\eta )$. Indeed, writing this term as
\begin{equation*}
R(y ) = \iint_{\bigcup_{\sigma \in [0,1]} \Gamma_{y,\sigma} } e^{\frac{i}{h} a(y,Z ) } \left(  \partial_{\overline{Z}} u_h + \frac{iu_h}{h} \partial_{\overline{Z}}a \right) \wedge \ud Z_1 \wedge \dots \wedge \ud Z_{2n}
\end{equation*} and computing $\partial_y^{\lambda} R(y )$ one has 
\begin{equation}
\partial_y^{\lambda} R(y ) = \sum \frac{\lambda!}{\mu! \nu!} \iint_{\bigcup_{\sigma \in [0,1]} \Gamma_{y,\sigma} } e^{\frac{i}{h} a(y,Z ) } \frac{\nu!}{\ell! \nu_1! \dots \nu_{\ell}!} \partial_y^{\nu_1} a \dots \partial_y^{\nu_\ell} a h^{-\ell} \partial_y^{\mu} g_h  \wedge \ud Z_1 \wedge \dots \wedge \ud Z_{2n} \label{eq:262}
\end{equation} where the sum runs over the set 
\begin{equation*}
\nu + \mu = \lambda, \qquad 
\nu_1 + \cdots \nu_{\ell} = \nu, \qquad 
\ell = 1, \cdots, \nu
\end{equation*} and $g_h =  \partial_{\overline{Z}} u_h + \frac{iu_h}{h} \partial_{\overline{Z}}a $. For $Z \in \bigcup_{\sigma \in [0,1]}$ one has 
\begin{equation*}
\Im a(y,Z ) \geq \frac{1}{C} |\Im Z|^2.
\end{equation*} Moreover, $\partial_{\overline{Z}}u_h$ and $\partial_{\overline{Z}}a$ are bounded by $C \exp \left( -\frac{1}{C} |\Im Z|^{-\frac{1}{s-1}}\right)$ as well as their derivatives by construction. We deduce estimates for the RHS of (\ref{eq:262}), since 
\begin{equation*}
\frac{|\partial_y^{\nu_1} a |}{\nu_1!}... \frac{|\partial_y^{\nu_l} a |}{\nu_l!}  \leq C^{|\nu| + 1}, \qquad  h|\partial_y^{\mu} g_h| \leq C^{1 + |\mu|} \mu!^s \exp(-\frac{1}{C} |\Im Z|^{-\frac{1}{s-1}} )  
\end{equation*} and 
\begin{equation*}
\left|   e^{\frac{i}{h} a(y,Z) }    \right| \leq e^{-\frac{1}{Ch}|\Im Z|^2 }  \sum_{\mu + \nu = \lambda} \frac{\lambda! \nu!}{\mu! \nu! \ell!} h^{-\ell}M! N! h^N |\Im Z|^{-2N} |\Im Z|^{\frac{M}{s-1}} \mu!^s,
\end{equation*} for all $N,M>0$. Choosing $N = \ell$ and ·$M=2 \ell(s-1)$ one has the bound 
\begin{equation*}
C^{1 + |\lambda|} \mu!^s \nu! \ell!^{2s-2} \leq C^{1+|\lambda|} \mu!^s \nu!^{2s-1} \leq C^{1 + |\lambda|} \lambda!^{2s-1}.
\end{equation*} Hence, $\partial_y^{\lambda}R(y)$ is bounded by 
\begin{equation*}
C^{1+|\lambda|} \lambda!^{2s-1} \exp( -\frac{1}{C} h^{\frac{-1}{2s-1}}   )
\end{equation*} so $R$ is a small $\cG^{2s-1}$ remainder.

\end{proof}

We use next Lemma \ref{lem:MelinSjostrandGevrey} to recover the asymptotic (\ref{eq:MelinSjostrand}). We can also check that on $\Gamma$ $P'a$ has the form 
\begin{equation*}
P'a = \frac{h}{i} \partial_{\Re x_1} a + \frac{h}{i} \frac{\partial p}{\partial \eta} (y,-\varphi_y') \partial_y a + ca + R'a
\end{equation*} with $c$ a symbol of order $-1$ and $R'$ a linear map $\cS_s^{m'} \rightarrow \cS_s^{m'-2}$. \par 

We solve inductively the equations on $\Gamma$:
\begin{equation*}
\frac{h}{i} v'a_0 + ca_0 = hb, \qquad \frac{h}{i} v'a_j + ca_j + R'a_{j-1} = 0, \quad j \geq 1,
\end{equation*} for 
\begin{equation*}
v' = \partial_{\Re x_1} + \frac{\partial p}{\partial \eta}(y,-\varphi_y')\partial_y + \overline{  \frac{\partial p}{\partial \eta}   }(y, -\varphi_y') \partial_{\overline{y}}.
\end{equation*}  We notice that $v'|_{\Gamma}$ is tangent to $\Gamma$ as a subset of $\comp^{2n}$. \par 

In order to conclude the proof, we need to check that the remainder 
\begin{equation*}
b'(x,y) = \frac{1}{h}e^{\frac{h}{i} \varphi  } \left( h D_{\Re x_1} - h^m P^t(y,D_y) \right)(a' e^{\frac{i}{h} \varphi} ) - b, 
\end{equation*} when $a' \sim \sum_{j\geq 0} a_j'$, $a_j' \in \cS_s^{-j}(\comp^n \times \comp^n)$ and $a_j'|_{\Gamma} = a_j $ for $a_j$ solving the transport equations above is small, i.e.:
$\left|b'(x,y)\right|\leq c\left\{ \exp(\frac{-1}{c}\left|y-y(x)\right|^{\frac{-1}{s-1}})+\exp(\frac{-1}{c}h^{\frac{-1}{s}})\right\} $

\color{black}

\newpage

\bibliographystyle{plain}                            

\begin{thebibliography}{99}

\bibitem{BedrossianMasmoudiMouhot}
J.Bedrossian, N.Masmoudi and C. Mouhot.
Landau Damping: Paraproducts and Gevrey Regularity,
\emph{Annals of PDE}, volume 2, 4 (2016). 





\bibitem{BoutetKree} 
L. Boutet de Monvel and P. Kr\'ee.
Pseudo-differential operators and Gevrey classes.
\emph{Ann. Inst. Fourier}, vol.7, 295-323 (1967).


\bibitem{Carleson} 
L. Carleson. 
On universal moment problems. 
\emph{Math. Scandinavica}, 9(1b): 197-206, 1961.



\bibitem{DuistermaatHormander}
J.J. Duistermaat and Lars Hörmander,
Fourier integral operators. II, 
\emph{Acta Math.}, Volume 128 (1972) no. 3-4, pp. 183-269. 


\bibitem{Egorov} 
Yu. V. Egorov. 
On canonical transformations of pseudo-differential operators, 
\emph{Uspechi Mat. Nauk},  25, 235-236, (1969).


\bibitem{Egorov84} 
Yu. V. Egorov. 
\emph{Linear Differential Equation of Principal Type}. Moscow: Nauka. English transl.: New
York: Con temp. Sov. Math. 1986, ZbI.574.35001.
\bibitem{EgorovBook} 
Yu. V. Egorov. 
\emph{Partial Differential Equations IV: Microlocal Analysis and Hyperbolic Equations},
Vol. 33 in Encyclopedia of Mathematical Sciences,  Springer (1993).






\bibitem{Eskin}
G. Eskin.
\emph{Lectures on linear partial differential equations}.
Volume 123 in Graduate Studies in Mathematics Series, 
American Mathematical Society, 2011.





\bibitem{GerardVaretMasmoudi}
D. G\'erard-Varet, Y. Maekawa and N. Masmoudi. 
Gevrey stability of Prandtl expansions for $2$-dimensional Navier–Stokes flows. 
\emph{Duke Math. J.} 167 (13) 2531 - 2631, 15 September 2018.


\bibitem{Gramchev}
T. Gramchev,
Classical pseudodifferential operators and Egorov's theorem in the Gevrey classes $G^s, s > 1$,
\emph{Russ. Math. Surv.} 40 125 (1985).



\bibitem{HitrikLascarSjostrandZerzeri}
M. Hitrik, R. Lascar, J. Sj\"ostrand and M. Zerzeri. 
Semiclassical gevrey operators in the complex domain. 
\emph{arXiv:2009.09125}, 2020, \emph{Ann. Inst. Fourier} (2023).


\bibitem{HitrikSjostrand}
M. Hitrik and J. Sj\"ostrand. 
Two minicourses on analytic microlocal analysis. 
In \emph{Algebraic and Analytic Microlocal Analysis}, pages 483-540. Springer, 2013.


\bibitem{Hormander} 
Lars H\"ormander. 
\emph{The analysis of linear partial differential operators}, Vols I-IV, Springer, 1983.

\bibitem{BLascar} 
Bernard Lascar. 
Propagation des singularit\'es Gevrey pour des op\'erateurs hyperboliques. 
\emph{Amer. J. of Math.}, 110 pages 413-449 (1988).


\bibitem{LascarLascar}
B. Lascar and R. Lascar. 
FBI transforms in Gevrey classes. 
\emph{Journal d'Analyse Math\'ematique}, 72(1):105-125, (1997).

\bibitem{LascarLascarMelrose}
B. Lascar, R. Lascar and R. Melrose.
Propagation des singularit\'es Gevrey pour la diffraction. 
\emph{Communications in partial differential equations}, 16(4-5):547-584, (1991).

\bibitem{LebeauGevrey3}
G. Lebeau. 
R\'egularit\'e Gevrey 3 pour la diffraction. 
\emph{Communications in Partial Differential Equations}, 9(15):1437-1494, (1984).



\bibitem{LernerBook}
N. Lerner.
\emph{Metrics on the phase space and non-selfadjoint pseudo-differential operators}, 
Pseudo-Differential Operators. Theory and Applications, Volume 3, Birkhäuser, 2010, xii+397 pages.

\bibitem{LernerEgorov}
Nicolas Lerner,
Sur deux contributions de Y. V. Egorov (1938–2018).
\emph{Annales de la Faculté des sciences de Toulouse : Mathématiques,} Serie 6, Volume 28 (2019) no. 1, pp. 1-9. 


\bibitem{MelinSjostrand}
A. Melin and J. Sj\"ostrand,
Fourier Integral Operators with complex phase functions. 
\emph{Lecture Notes in Mathematics,} vol. 459, pp. 120-224 (1975). 
Springer.


\bibitem{MouhotVillani}
C. Mouhot and C. Villani.
On Landau damping.
\emph{Acta Mathematica}, volume 207, pages 29–201 (2011).



\bibitem{SjostrandAsterisque} 
J. Sj\"ostrand. 
\emph{Singularit\'es analytiques microlocales.} Ast\'erisque 1982.

\bibitem{Zworski}
M. Zworski. 
\emph{Semiclassical analysis}, 
volume 138. American Math. Soc., 2012.



\end{thebibliography}
{}

\end{document}